\documentclass{birkjour}
\usepackage[noadjust]{cite}

\newcommand{\BR}{\mathbb{R}}

\newcommand{\ed}{\end{document}}

\usepackage{mathtools,amssymb,graphicx}
 \newtheorem{thm}{Theorem}[section]
 \newtheorem{cor}[thm]{Corollary}
 
 \newtheorem{prop}[thm]{Proposition}
 \theoremstyle{definition}
 \newtheorem{defn}[thm]{Definition}
 \theoremstyle{remark}
 \newtheorem{rem}[thm]{Remark}
 \newtheorem*{ex}{Example}
 \numberwithin{equation}{section}
\allowdisplaybreaks
\usepackage{xcolor}

\begin{document}

%
%
%
%
%
%
%
%
%

\title[Potential Vector Fields in $\mathbb R^4$]
 {Potential Vector Fields in $\mathbb R^4$ and New Generalizations of the Cauchy-Riemann System}

\author[Dmitry Bryukhov]{Dmitry Bryukhov}

\address{
 \texttt{https://orcid.org/0000-0002-8977-3282} \\
Science City Fryazino, Russia}
 \email{bryukhov@mail.ru}

\subjclass{Primary 30G35, 35J46; Secondary 30C65, 35Q99, 34A34}

\keywords{Generalizations of the Cauchy-Riemann system, Potential
meridional fields in $\mathbb R^4$, Four-dimensional
$\alpha$-meridional mappings of the second kind, The radially
holomorphic potential in $\mathbb R^4$, Gradient dynamical systems}


\begin{abstract}
This paper extends approach of recent author's paper devoted to
special classes of exact solutions of the static Maxwell system in
inhomogeneous isotropic media and new generalizations of the
Cauchy-Riemann system  in $\mathbb R^3$. Two families of
generalizations of the Cauchy-Riemann system with variable
coefficients in $\mathbb R^4$ are presented in the context of
\emph{Non-Euclidean geometry}. Analytic models of a wide range of
potential meridional vector fields in $\mathbb R^4$ are
characterized using a family of Vekua type systems in cylindrical
coordinates. The specifics of meridional fields allows us to
introduce the concept of four-dimensional $\alpha$-meridional
mappings of the first and second kind depending on the values of a
real parameter $\alpha$. In case $\alpha=2$ tools of the radially
holomorphic potential in $\mathbb R^4$ are developed in the context
of \emph{Generalized axially symmetric potential theory (GASPT)}.
Analytic models of potential meridional fields in $\mathbb R^4$
generated by Bessel functions of the first kind of integer order and
quaternionic argument are described in case $\alpha=2$. In case
$\alpha =0$ the geometric specifics of four-dimensional harmonic
meridional mappings of the second kind is demonstrated explicitly in
the context of the theory of \emph{Gradient dynamical systems with
harmonic potential}.
\end{abstract}

\maketitle
\section{Introduction, Preliminaries, and Notation}

\subsection{Introduction}
A rich variety of analytic models of potential vector fields $\vec
V= (V_0, V_1, V_2, V_3)$ in $\mathbb R^4=\{(x_0, x_1, x_2, x_3)\}$
is described by the following first-order static system including a
variable $C^1$-coefficient $\phi= \phi(x_0,x_1,x_2,x_3)>0$:
\begin{gather}
\begin{cases}
 \mathrm{div} \, (\phi \  \vec V) =0,  \\[1ex]
    \frac{\partial{V_0}}{\partial{x_1}}= \frac{\partial{V_1}}{\partial{x_0}},
       \quad  \frac{\partial{V_0}}{\partial{x_2}}= \frac{\partial{V_2}}{\partial{x_0}},
       \quad \frac{\partial{V_0}}{\partial{x_3}}= \frac{\partial{V_3}}{\partial{x_0}}, \\[1ex]
      \frac{\partial{V_1}}{\partial{x_2}}= \frac{\partial{V_2}}{\partial{x_1}},
       \quad \frac{\partial{V_1}}{\partial{x_3}}= \frac{\partial{V_3}}{\partial{x_1}},
       \quad \frac{\partial{V_2}}{\partial{x_3}}=
       \frac{\partial{V_3}}{\partial{x_2}}.
\end{cases}
\label{conservative-system-4}
\end{gather}
 In our setting the Euclidean space $\mathbb R^4$ involves the longitudinal
variable $x_0$ and the cylindrical radial variable $ \rho = \sqrt
{x_1^2+x_2^2+x_3^2}$. The scalar potential $h= h(x_0,x_1,x_2,x_3)$
in simply connected open domains $\Lambda \subset \mathbb R^4$,
where $\vec V = \mathrm{grad} \ h$, allows us to reduce every
$C^1$-solution of the system~\eqref{conservative-system-4} to a
$C^2$-solution of the continuity equation
\begin{gather}
\mathrm{div} \, (\phi \ \mathrm{grad}{\ h}) = 0.
\label{Liouville-eq-4}
\end{gather}

The space $\mathbb R^4$ and the coefficient $\phi=
\phi(x_0,x_1,x_2,x_3)$ may be interpreted as the phase space $X$ (or
the state space)  and the Jacobi last multiplier $M=
M(x_0,x_1,x_2,x_3)$, respectively, in the context of the theory of
\emph{Gradient dynamical systems}, where $\vec V = \frac {d{\vec
x}}{dt} = \mathrm{grad} \ h$ (see, e.g., \cite
{NemStep:1960,Hirsch:1984,Perko:2001,HirschSmaleDevaney:2004,Sinai-Princeton:1994,NucciTam:2012}).

 Some types of harmonic potential vector fields have been extensively
investigated by Brackx, Delanghe, Sommen et al. in the context of
\emph{Quaternionic analysis} and the theory of
\emph{Quaternion-valued monogenic functions} since the 1970s (see,
e.g.,
\cite{Sudbery:1979,BraDelSom:1982,DelSomSouc:1992,Lounesto:2001,BraDel:2003,Del:2007}).
 Remarkable extensions of the basic concepts of power series, holomorphicity,
elementary and special functions have been developed for the past 40
years, in particular, in the context of \emph{Quaternionic analysis
and its modifications} (see, e.g.,
\cite{Leut:CV17,HempLeut:1996,LeZe:CMFT2004,ErOrel:2019,Leut:2018-AACA,Leut:2020-CAOT}),
of the theory of \emph{Holomorphic functions in $n$-dimensional
space} (see, e.g., \cite{GuHaSp:2008,GuHaSp:2016}), of the theory of
\emph{Cullen regular (also referred to as ``slice regular", ``slice
monogenic" or ``slice hyperholomorphic" in various settings)
functions} (see, e.g.,
\cite{GentStruppa:2006,ColSabStr:2009,ColG-CerSab:2011,ColSabStr:2011,Stoppato:2011,GenStop:2012,ColSabStr:2016,GalG-CerSab:2015,GenStop:2021}.

 Essentially new tools of quaternionic power series with quaternion-valued
coefficients have been developed by Leutwiler and Hempfling in the
1990s to investigate analytic properties of exact solutions of the
system
\begin{gather*}
(H_4)
\begin{cases}
      x_3 \left( \frac{\partial{u_0}}{\partial{x_0}}-
      \frac{\partial{u_1}}{\partial{x_1}}-\frac{\partial{u_2}}{\partial{x_2}}-
      \frac{\partial{u_3}}{\partial{x_3}} \right) + 2 u_3 = 0 \\[1ex]
      \frac{\partial{u_0}}{\partial{x_1}}=-\frac{\partial{u_1}}{\partial{x_0}},
      \quad  \frac{\partial{u_0}}{\partial{x_2}}=-\frac{\partial{u_2}}{\partial{x_0}},
       \quad \frac{\partial{u_0}}{\partial{x_3}}=-\frac{\partial{u_3}}{\partial{x_0}}, \\[1ex]
      \frac{\partial{u_1}}{\partial{x_2}}=\ \ \frac{\partial{u_2}}{\partial{x_1}},
      \quad  \frac{\partial{u_1}}{\partial{x_3}}=\ \ \frac{\partial{u_3}}{\partial{x_1}},
       \quad  \frac{\partial{u_2}}{\partial{x_3}}=\ \
       \frac{\partial{u_3}}{\partial{x_2}}
 \end{cases}
\end{gather*}
in the context of \emph{Modified quaternionic analysis in $\mathbb
R^4$}, where $x_3>0$ (see, e.g.,
\cite{Leut:CV17,HempLeut:1996,LeZe:CMFT2004}). This system is called
the system $(H_4)$ in honor of Hodge.

General class of exact solutions of the system $(H_4)$ may be
equivalently represented as general class of exact solutions of the
static system
\begin{gather}
\begin{cases}
 \mathrm{div} \, (x_3^{-2} \ \vec V) =0,  \\[1ex]
    \frac{\partial{V_0}}{\partial{x_1}}= \frac{\partial{V_1}}{\partial{x_0}},
       \quad \frac{\partial{V_0}}{\partial{x_2}}= \frac{\partial{V_2}}{\partial{x_0}},
       \quad \frac{\partial{V_0}}{\partial{x_3}}= \frac{\partial{V_3}}{\partial{x_0}}, \\[1ex]
      \frac{\partial{V_1}}{\partial{x_2}}= \frac{\partial{V_2}}{\partial{x_1}},
       \quad \frac{\partial{V_1}}{\partial{x_3}}= \frac{\partial{V_3}}{\partial{x_1}},
       \quad \frac{\partial{V_2}}{\partial{x_3}}= \frac{\partial{V_3}}{\partial{x_2}},
\end{cases}
 \label{2-hyperbolic-isotropic-system-4}
\end{gather}
 where $\phi= \phi(x_3) = x_3^{-2}$;
 $\ (V_0, V_1, V_2, V_3)=(u_0, -u_1, -u_2, -u_3)$.

Assume that $ \rho > 0$.
 Axially symmetric Fueter's construction in $\mathbb R^4$ (see, e.g., \cite{Fueter:1934,Leut:CV17})
\begin{gather}
  F = F(x) = u_0 + i u_1 + j u_2 + ku_3 = u_0(x_0, \rho) + I \ u_{\rho}(x_0, \rho),\notag\\
x = x_0 + I \rho, \quad I = \frac{i x_1+ j x_2 + kx_3}{\rho}, \quad
i^2 = j^2 = k^2 = I^2=-1,
\label{Fueter-4}
\end{gather}
 was characterized  by the author in 2003 \cite{Br:2003} as joint class of analytic solutions of the system
$(H_4)$ and axially symmetric generalization of the Cauchy-Riemann
system
\begin{gather}
\begin{cases}
      (x_1^2+x_2^2+x_3^2)\left( \frac{\partial{u_0}}{\partial{x_0}}-
      \frac{\partial{u_1}}{\partial{x_1}}-\frac{\partial{u_2}}{\partial{x_2}}-
      \frac{\partial{u_3}}{\partial{x_3}} \right) + 2(x_1u_1+x_2u_2+x_3u_3)=0 \\[1ex]
      \frac{\partial{u_0}}{\partial{x_1}}=-\frac{\partial{u_1}}{\partial{x_0}},
        \quad \frac{\partial{u_0}}{\partial{x_2}}=-\frac{\partial{u_2}}{\partial{x_0}},
        \quad \frac{\partial{u_0}}{\partial{x_3}}=-\frac{\partial{u_3}}{\partial{x_0}}, \\[1ex]
      \frac{\partial{u_1}}{\partial{x_2}}=\ \ \frac{\partial{u_2}}{\partial{x_1}},
        \quad \frac{\partial{u_1}}{\partial{x_3}}=\ \ \frac{\partial{u_3}}{\partial{x_1}},
        \quad \frac{\partial{u_2}}{\partial{x_3}}=\ \ \frac{\partial{u_3}}{\partial{x_2}}
 \end{cases}
 \label{Bryukhov-4}
\end{gather}
 under conditions of
\begin{equation}
{u_1}{x_2}={u_2}{x_1}, \quad {u_1}{x_3}={u_3}{x_1}, \quad
{u_2}{x_3}={u_3}{x_2}, \quad x_3 > 0.
\label{spec.cond-4}
\end{equation}

 General class of exact solutions of the static system
\begin{gather}
\begin{cases}
  \mathrm{div} \, (\rho^{-2} \ \vec V) =0,  \\[1ex]
    \frac{\partial{V_0}}{\partial{x_1}}= \frac{\partial{V_1}}{\partial{x_0}},
       \quad \frac{\partial{V_0}}{\partial{x_2}}= \frac{\partial{V_2}}{\partial{x_0}},
       \quad \frac{\partial{V_0}}{\partial{x_3}}= \frac{\partial{V_3}}{\partial{x_0}}, \\[1ex]
      \frac{\partial{V_1}}{\partial{x_2}}= \frac{\partial{V_2}}{\partial{x_1}},
       \quad \frac{\partial{V_1}}{\partial{x_3}}= \frac{\partial{V_3}}{\partial{x_1}},
       \quad \frac{\partial{V_2}}{\partial{x_3}}= \frac{\partial{V_3}}{\partial{x_2}},
\end{cases}
 \label{2-axial-isotropic-system-4}
\end{gather}
is equivalently represented as general class of exact solutions of
the system~\eqref{Bryukhov-4}, where $\phi= \phi(\rho) = \rho^{-2}$;
 $\ (V_0, V_1, V_2, V_3)=(u_0, -u_1, -u_2, -u_3)$.

Surprisingly, applications of the system $(H_4)$, the system
~\eqref{Bryukhov-4} and Fueter's construction in $\mathbb
R^4$~\eqref{Fueter-4} in accordance with the
systems~\eqref{2-hyperbolic-isotropic-system-4},~\eqref{2-axial-isotropic-system-4}
 have been missed.

The main goal of this paper is to develop new tools of the theory of
\emph{Potential meridional vector fields in $\mathbb R^4$} by means
of different generalizations of the Cauchy-Riemann system with
variable coefficients.

The paper is organized as follows. In Section 2, two families of
generalizations of the Cauchy-Riemann system with variable
coefficients in $\mathbb R^4$ are provided. Analytic properties of
potential meridional fields are considered in the context of
\emph{Modified quaternionic analysis in $\mathbb R^4$},
\emph{Hyperbolic function theory in the skew-field of quaternions}
and the theory of \emph{Modified harmonic functions in $\mathbb
R^4$}. In Section 3, new concept of four-dimensional
$\alpha$-meridional mappings of the first and second kind, where
$\alpha \in \BR$, is introduced. In Section 4, in case $\alpha=2$
tools of the radially holomorphic potential in $\mathbb R^4$ are
developed using the concept of radially holomorphic functions in
$\mathbb R^4$ introduced by G\"{u}rlebeck, Habetha, Spr\"{o}{\ss}ig
in the context of the theory of \emph{Holomorphic functions in
n-dimensional space}. In Section 5, new analytic models of potential
meridional fields in $\mathbb R^4$ generated by the quaternionic
Fourier-Fueter cosine and sine transforms of real-valued original
functions are described in case $\alpha=2$. As a corollary, integral
representations of Bessel functions of the first kind of integer
order and quaternionic argument are obtained. In Section 6, in case
$\alpha=0$ the geometric specifics of four-dimensional harmonic
meridional mappings of the second kind is demonstrated explicitly in
the context of the theory of \emph{Gradient dynamical systems with
harmonic potential}, where $\vec V= \frac {d{\vec x}}{dt} =
\mathrm{grad} \ h$, $ \ \Delta \ h= 0$.

\subsection{Preliminaries}
 New families of generalizations of the Cauchy-Riemann system  with variable coefficients in $\mathbb R^4$
 may be provided by means of the following first-order system:
\begin{gather}
\begin{cases}
 \phi \left( \frac{\partial{u_0}}{\partial{x_0}}-
      \frac{\partial{u_1}}{\partial{x_1}}- \frac{\partial{u_2}}{\partial{x_2}} - \frac{\partial{u_3}}{\partial{x_3}} \right)
      + \left(\frac{\partial{ \phi}}{\partial{x_0}}u_0
       - \frac{\partial{\phi}}{\partial{x_1}}u_1 - \frac{\partial{\phi}}{\partial{x_2}}u_2 - \frac{\partial{\phi}}{\partial{x_3}}u_3 \right) = 0,  \\[1ex]
       \frac{\partial{u_0}}{\partial{x_1}}=-\frac{\partial{u_1}}{\partial{x_0}},
      \quad \frac{\partial{u_0}}{\partial{x_2}}=-\frac{\partial{u_2}}{\partial{x_0}},
       \quad \frac{\partial{u_0}}{\partial{x_3}}=-\frac{\partial{u_3}}{\partial{x_0}}, \\[1ex]
      \frac{\partial{u_1}}{\partial{x_2}}=\ \ \frac{\partial{u_2}}{\partial{x_1}},
        \quad \frac{\partial{u_1}}{\partial{x_3}}=\ \ \frac{\partial{u_3}}{\partial{x_1}},
        \quad \frac{\partial{u_2}}{\partial{x_3}}=\ \ \frac{\partial{u_3}}{\partial{x_2}}.
\end{cases}
\label{Bryukhov-general-4}
\end{gather}
Suppose that $ (V_0, V_1, V_2, V_3)=(u_0, -u_1, -u_2, -u_3)$.
General class of $C^1$-solutions of the
system~\eqref{conservative-system-4} may be equivalently represented
as general class of $C^1$-solutions of the
system~\eqref{Bryukhov-general-4}.

The continuity equation~\eqref{Liouville-eq-4} in the expanded form
is expressed as
\begin{gather}
\phi \Delta h
 + \frac{\partial{\phi}}{\partial{x_0}} \frac{\partial{h}}{\partial{x_0}} + \frac{\partial{\phi}}{\partial{x_1}} \frac{\partial{h}}{\partial{x_1}} +
  \frac{\partial{\phi}}{\partial{x_2}}\frac{\partial{h}}{\partial{x_2}} + \frac{\partial{\phi}}{\partial{x_3}} \frac{\partial{h}}{\partial{x_3}}
  =0,
\label{Liouville-eq-4-expanded}
\end{gather}
where the Laplacian $ \ \Delta :=
\frac{{\partial}^2{}}{{\partial{x_0}}^2} +
\frac{{\partial}^2{}}{{\partial{x_1}}^2} +
\frac{{\partial}^2{}}{{\partial{x_2}}^2} +
\frac{{\partial}^2{}}{{\partial{x_3}}^2}$.

The equation
\begin{equation}
h(x_0,x_1,x_2,x_3) = C = const
 \label{equipotential}
\end{equation}
allows us to establish important properties of the equipotential
hypersurfaces in simply connected open domains $\Lambda \subset
\mathbb R^4$. Using the total differential $dh$,
Eq.~\eqref{equipotential} is reformulated as an exact differential
equation (see, e.g., \cite{Walter:1998})
\begin{gather*}
 dh  = \frac{\partial{h}}{\partial{x_0}} d{x_0} + \frac{\partial{h}}{\partial{x_1}} d{x_1}
 + \frac{\partial{h}}{\partial{x_2}} d{x_2} + \frac{\partial{h}}{\partial{x_3}} d{x_3} = 0.
\end{gather*}

Let $\varsigma$ be a real independent variable. Assume that
homogeneous first-order partial differential equation
\begin{equation}
  \frac{\partial{h}}{\partial{x_0}} W_0 + \frac{\partial{h}}{\partial{x_1}} W_1 + \frac{\partial{h}}{\partial{x_2}} W_2
   + \frac{\partial{h}}{\partial{x_3}} W_3 = 0
\label{PDE}
\end{equation}
 is satisfied in $ \Lambda$ such that
\begin{equation}
W_l(x_0,x_1,x_2,x_3) = \frac{dx_l}{d\varsigma} \quad (l = 0,1,2,3).
\label{characteristics-4}
\end{equation}

The system $(\ref{characteristics-4})$ allows us to introduce the
characteristic vector field $ \vec W = (W_0, W_1, W_2, W_3)$ for
equation~\eqref{PDE} in $\Lambda$ in the context of
\emph{Geometrical methods of the theory of ordinary differential
equations} (see, e.g., \cite{ArnoldGeom}). The scalar potential $h=
h(x_0,x_1,x_2,x_3)$ is accordingly referred to as a first integral
of the characteristic vector field $\vec W = (W_0, W_1, W_2, W_3)$
(or of its associated system $\frac{dx_0}{W_0} = \frac{dx_1}{W_1} =
\frac{dx_2}{W_2} = \frac{dx_3}{W_3}$) in $\Lambda$ if and only if
equation~\eqref{PDE} is satisfied in $\Lambda$ (see, e.g.,
\cite{ZachThoe:1986,ArnoldGeom,NucciTam:2012}).

Equation~\eqref{PDE} is geometrically characterized as the
orthogonality condition for vector fields $\vec V$ and $\vec W$:
\begin{gather}
  ( \vec V, \vec W ) = (\mathrm{grad} \ h, \vec W ) = 0.
\label{orthogonality-Maxwell-electric}
\end{gather}
Equation~\eqref{orthogonality-Maxwell-electric} is satisfied, in
particular, under condition of $ \vec V = \mathrm{grad} \ h$ $=
(u_0, -u_1, -u_2, -u_3) = 0$.
\begin{defn}
Let $\Lambda \subset \mathbb R^4$ be a simply connected open domain.
Every point $x^*=(x_0^*,x_1^*,x_2^*,x_3^*) \in \Lambda$ under
condition of $ \mathrm{grad} \ h(x^*) =0$ is called a critical point
of the scalar potential $h$ in $\Lambda$. The set of critical points
is called the critical set of $h$ in $\Lambda$.
 \end{defn}
Geometric and topological properties of the critical sets of the
scalar potential $h$ are of particular interest to \emph{Catastrophe
theory} (see, e.g., \cite{Gilmore}).

  The Hessian matrix $\mathbf{H}(h(x))$ of the scalar potential $h$
   may be considered as the Jacobian matrix $\mathbf{J}(\vec V(x))$ of the vector field $\vec V$,
  where $J_{l m} = \frac{\partial{V_l}}{\partial{x_m}} \ $ $ ( l, m = 0,1,2,3)$:
\begin{equation}
\left(
\begin{array}{rrrr}
  \frac{\partial{V_0}}{\partial{x_0}} & \frac{\partial{V_0}}{\partial{x_1}} & \frac{\partial{V_0}}{\partial{x_2}} & \frac{\partial{V_0}}{\partial{x_3}} \\[1ex]
 \frac{\partial{V_1}}{\partial{x_0}}  & \frac{\partial{V_1}}{\partial{x_1}}  &  \frac{\partial{V_1}}{\partial{x_2}} &  \frac{\partial{V_1}}{\partial{x_3}} \\[1ex]
 \frac{\partial{V_2}}{\partial{x_0}}  & \frac{\partial{V_2}}{\partial{x_1}}  &  \frac{\partial{V_2}}{\partial{x_2}} & \frac{\partial{V_2}}{\partial{x_3}} \\[1ex]
 \frac{\partial{V_3}}{\partial{x_0}} & \frac{\partial{V_3}}{\partial{x_1}} & \frac{\partial{V_3}}{\partial{x_2}} & \frac{\partial{V_3}}{\partial{x_3}}
 \end{array}
\right) = \left(
\begin{array}{rrrr}
 \ \ \frac{\partial{u_0}}{\partial{x_0}} &  \ \ \frac{\partial{u_0}}{\partial{x_1}} & \ \ \frac{\partial{u_0}}{\partial{x_2}} & \ \ \frac{\partial{u_0}}{\partial{x_2}} \\[1ex]
 -\frac{\partial{u_1}}{\partial{x_0}}  & -\frac{\partial{u_1}}{\partial{x_1}}  &  -\frac{\partial{u_1}}{\partial{x_2}}  &  -\frac{\partial{u_1}}{\partial{x_2}} \\[1ex]
 -\frac{\partial{u_2}}{\partial{x_0}}  & -\frac{\partial{u_2}}{\partial{x_1}}  &  -\frac{\partial{u_2}}{\partial{x_2}} & -\frac{\partial{u_2}}{\partial{x_2}} \\[1ex]
 -\frac{\partial{u_3}}{\partial{x_0}}  & -\frac{\partial{u_3}}{\partial{x_1}}  &  -\frac{\partial{u_3}}{\partial{x_2}} &  -\frac{\partial{u_3}}{\partial{x_2}}
 \end{array}
\right)
\label{Hessian-matrix-4}
\end{equation}
 The characteristic equation of the Hessian matrix
 $\mathbf{H}(h(x))$ in the general four-dimensional setting is expressed as
\begin{gather}
\lambda^4 - I_{\mathbf{J}(\vec V)} \lambda^3 + II_{\mathbf{J}(\vec
V)} \lambda^2 - III_{\mathbf{J}(\vec V)} \lambda +
IV_{\mathbf{J}(\vec V)}
  = 0.
  \label{characteristic lambda-4}
\end{gather}
The principal invariants $I_{\mathbf{J}(\vec V)}$,
$II_{\mathbf{J}(\vec V)}$, $III_{\mathbf{J}(\vec V)}$,
$IV_{\mathbf{J}(\vec V)}$ of the Jacobian
matrix~\eqref{Hessian-matrix-4} are given by formulas

\begin{gather*}
\begin{cases}
I_{\mathbf{J}(\vec V)} = \lambda_0 +
\lambda_1 + \lambda_2 + \lambda_3 = J_{00} + J_{11} + J_{22} + J_{33},  \\[1ex]
II_{\mathbf{J}(\vec V)} = \lambda_0 \lambda_1 + \lambda_0 \lambda_2
+ \lambda_0 \lambda_3 + \lambda_1 \lambda_2 + \lambda_1 \lambda_3 +
\lambda_2 \lambda_3 = \\[1ex]
J_{00}J_{11} + J_{00}J_{22} +J_{00}J_{33} +J_{11}J_{22} +J_{11}J_{33}+ J_{22}J_{33}, \\[1ex]
III_{\mathbf{J}(\vec V)} = \lambda_0 \lambda_1 \lambda_2 + \lambda_0
\lambda_1 \lambda_3 + \lambda_0 \lambda_2 \lambda_3 + \lambda_1
\lambda_2 \lambda_3 = \\[1ex]
  J_{00}J_{11}J_{22} + J_{00}J_{11}J_{33} + J_{00}J_{22}J_{33} + J_{11}J_{22}J_{33} \\[1ex]
 + 2J_{01}J_{02}J_{12} + 2J_{01}J_{03}J_{13} + 2J_{02}J_{03}J_{23} + 2J_{12}J_{13}J_{23}  \\[1ex]
  - J_{00}(J_{12})^2- J_{00}(J_{13})^2- J_{00}(J_{23})^2 - J_{11}(J_{02})^2- J_{11}(J_{03})^2- J_{11}(J_{23})^2 \\[1ex]
  - J_{22}(J_{01})^2- J_{22}(J_{03})^2- J_{22}(J_{13})^2 - J_{33}(J_{01})^2- J_{33}(J_{02})^2- J_{33}(J_{12})^2,  \\[1ex]
IV_{\mathbf{J}(\vec V)} = \lambda_0 \lambda_1 \lambda_2 \lambda_3 =
J_{00}J_{11}J_{22}J_{33} + 2J_{00}J_{12}J_{13}J_{23} + 2J_{01}J_{02}J_{12}J_{33}   \\[1ex]
  + 2J_{02}J_{03}J_{23}J_{11} + (J_{01}J_{23})^2 + 2J_{01}J_{03}J_{13}J_{22} + (J_{02}J_{13})^2 + (J_{03}J_{12})^2 \\[1ex]
 - 2J_{01}J_{03}J_{12}J_{23} -2J_{01}J_{02}J_{13}J_{23} - 2J_{02}J_{03}J_{12}J_{13} - J_{00}J_{11}(J_{23})^2 - \\[1ex]
J_{00}J_{22}(J_{13})^2 -J_{00}J_{33}(J_{12})^2 -
J_{11}J_{22}(J_{03})^2 -J_{11}J_{33}(J_{02})^2
-J_{22}J_{33}(J_{01})^2.
\end{cases}
\end{gather*}

\begin{defn}
Every point $x \in \Lambda$ under condition of $\det\mathbf{J}(\vec
V(x)) =0$ is called a degenerate point of the Jacobian matrix
$\mathbf{J}(\vec V(x))$ in $\Lambda$.
 \end{defn}

 Meanwhile, static potential vector fields in $\mathbb R^4$  may be investigated
in the context of \emph{Non-Euclidean geometry} using the
Laplace-Beltrami equation  (see, e.g.,
\cite{Ahlfors:1981,Leut:CV17,HempLeut:1996,BrKaeh:2016,Br:Hefei2020})
\begin{gather*}
 \Delta_{B} \ h := \phi^{-2} \mathrm{div}( \phi \ \mathrm{grad}{\ h}) = 0
\end{gather*}
with respect to the conformal metric
 \begin{gather}
 ds^2 = \phi (d{x_0}^2 + d{x_1}^2 + d{x_2}^2 + d{x_3}^2).
\label{Riemannian conformal metric-4}
\end{gather}

Euclidean geometry is provided in case $\phi = const$. Harmonic
potential fields $\vec V$ in $\mathbb R^4$ satisfy the first-order
static system
\begin{gather*}
\begin{cases}
       \frac{\partial{V_0}}{\partial{x_0}} + \frac{\partial{V_1}}{\partial{x_1}} + \frac{\partial{V_2}}{\partial{x_2}} + \frac{\partial{V_3}}{\partial{x_3}} =0,  \\[1ex]
    \frac{\partial{V_0}}{\partial{x_1}}= \frac{\partial{V_1}}{\partial{x_0}},
       \quad \frac{\partial{V_0}}{\partial{x_2}}= \frac{\partial{V_2}}{\partial{x_0}},
       \quad \frac{\partial{V_0}}{\partial{x_3}}= \frac{\partial{V_3}}{\partial{x_0}}, \\[1ex]
      \frac{\partial{V_1}}{\partial{x_2}}= \frac{\partial{V_2}}{\partial{x_1}},
       \quad \frac{\partial{V_1}}{\partial{x_3}}= \frac{\partial{V_3}}{\partial{x_1}},
       \quad \frac{\partial{V_2}}{\partial{x_3}}=
       \frac{\partial{V_3}}{\partial{x_2}}.
\end{cases}
\end{gather*}
This system is well-known in the context of \emph{Fourier analysis
on Euclidean spaces} as the Riesz system (see, e.g.,
\cite{SteinWeiss:1971}). On the other hand, harmonic potential
fields in $\mathbb R^4$ may be equivalently represented as general
class of analytic solutions of the system
\begin{gather*}
(R_4)
\begin{cases}
 \frac{\partial{u_0}}{\partial{x_0}}-
      \frac{\partial{u_1}}{\partial{x_1}}- \frac{\partial{u_2}}{\partial{x_2}} - \frac{\partial{u_3}}{\partial{x_3}} =0,  \\[1ex]
       \frac{\partial{u_0}}{\partial{x_1}}=-\frac{\partial{u_1}}{\partial{x_0}},
     \quad \frac{\partial{u_0}}{\partial{x_2}}=-\frac{\partial{u_2}}{\partial{x_0}},
       \quad \frac{\partial{u_0}}{\partial{x_3}}=-\frac{\partial{u_3}}{\partial{x_0}}, \\[1ex]
      \frac{\partial{u_1}}{\partial{x_2}}=\ \ \frac{\partial{u_2}}{\partial{x_1}},
       \quad \frac{\partial{u_1}}{\partial{x_3}}=\ \ \frac{\partial{u_3}}{\partial{x_1}},
       \quad \frac{\partial{u_2}}{\partial{x_3}}=\ \ \frac{\partial{u_3}}{\partial{x_2}},
\end{cases}
\end{gather*}
where  $(u_0, u_1, u_2, u_3) :=(V_0, -V_1, -V_2, -V_3)$. This system
is called the system $(R_4)$ in honor of Riesz. Four-dimensional
harmonic mappings $u= u_0 + iu_1 + ju_2 + ju_3: \Lambda \rightarrow
\mathbb{R}^4$ in the context of \emph{Quaternionic analysis} are
referred to as quaternion-valued monogenic functions (see, e.g.,
\cite{BraDel:2003,LeZe:CMFT2004,Del:2007}).

The system~\eqref{Bryukhov-general-4} in the context of
\emph{Modified quaternionic analysis in $\mathbb R^4$} (see, e.g.,
\cite{Leut:CV17,HempLeut:1996,LeZe:CMFT2004}), \emph{Hyperbolic
function theory in the skew-field of quaternions} (see, e.g.,
\cite{ErOrel:2019}) and the theory of \emph{Modified harmonic
functions in $\mathbb R^4$} (see, e.g.,
\cite{Leut:2018-AACA,Leut:2020-CAOT}) may be characterized as
generalized non-Euclidean modification of the system $(R_4)$ with
respect to the conformal metric~\eqref{Riemannian conformal
metric-4}.

\subsection{Notation}
The real algebra of quaternions $\mathbb H$ is a four dimensional
skew algebra over the real field generated by real unity $1$. Three
imaginary unities $i, j,$ and $k$ satisfy to multiplication rules
\begin{gather*}
i^2 = j^2 = k^2  = ijk = -1, \quad ij = -ji = k.
\end{gather*}

The independent quaternionic variable is defined as $$x = x_0 + ix_1  + jx_2  + kx_3.$$

The quaternion conjugation of $x$ is defined by the following automorphism:
$$ x \mapsto \overline{x} := x_0 - ix_1 - jx_2 - kx_3.$$

In such way, we deal with the Euclidean norm  in $\mathbb R^4$
$$
\| x \|^2 :=  x \overline{x} = x_0^2 + x_1^2 + x_2^2 + x_3^2 := r^2,
$$
and the identification
$$
x = x_0 + ix_1  + jx_2  + kx_3 \sim (x_0,  x_1,  x_2,  x_3)
$$
between $\mathbb H$ and $\mathbb R^4$ is valid. Moreover, for every non-zero value of $x$ an unique inverse value exists: $x^{-1} = \overline{x} / \| x \|^2.$

The dependent quaternionic variable is defined as
$$
u = u_0 + iu_1 + ju_2 +  ju_3 \sim (u_0, u_1, u_2, u_3).
$$

The quaternion conjugation of $u$ is defined by the following automorphism:
$$
u \mapsto \overline{u} := u_0 - iu_1 - ju_2 - ku_3.
$$

Assume that $x_3 > 0$. In cylindrical coordinates in $\mathbb{R}^4$
we obtain

 $x = x_0 + \rho (i\cos{\theta} + j \sin{\theta}\cos{\psi} + k\sin{\theta}\sin{\psi}),$ where

$x_1 = \rho \cos{\theta}, \quad x_2 = \rho \sin{\theta}\cos{\psi},
\quad x_3 = \rho \sin{\theta}\sin{\psi},$

 $ \varphi=  \arccos \frac{x_0}{r}\ \ (0 < \varphi < \pi),
\quad \theta = \arccos \frac{x_1}{\rho}\ \ (0 \leq \theta \leq
2\pi),$

$ \psi = \mathrm{arccot} \frac{x_2}{x_3}\ \ (0 < \psi < \pi).$

\begin{defn}
Let $\Omega\subset \mathbb R^4$ be an open set. Every continuously
differentiable mapping  $u= u_0 + iu_1 + ju_2 + ju_3: \Omega
\rightarrow \mathbb{R}^4$ is called quaternion-valued $C^1$-function
 in $\Omega$.
\end{defn}

\section{Two Types of Potential Vector Fields in $\mathbb R^4$
 and Criterions of Potential Meridional Fields}

To provide the first type of potential vector fields in $\mathbb
R^4$, assume that  $C^1$-coefficient $\phi(x_0,x_1,x_2,x_3)$ depends
only on variable $x_3$ so that $\phi= \phi(x_3)
> 0$. The system $(\ref{Bryukhov-general-4})$ is described as
\begin{gather}
\begin{cases}
  \phi(x_3) \left( \frac{\partial{u_0}}{\partial{x_0}}-
      \frac{\partial{u_1}}{\partial{x_1}}- \frac{\partial{u_2}}{\partial{x_2}} - \frac{\partial{u_3}}{\partial{x_3}} \right)
      - \frac{d{\phi}}{d{x_3}}u_3 = 0,  \\[1ex]
       \frac{\partial{u_0}}{\partial{x_1}}=-\frac{\partial{u_1}}{\partial{x_0}},
     \quad \frac{\partial{u_0}}{\partial{x_2}}=-\frac{\partial{u_2}}{\partial{x_0}},
       \quad \frac{\partial{u_0}}{\partial{x_3}}=-\frac{\partial{u_3}}{\partial{x_0}}, \\[1ex]
      \frac{\partial{u_1}}{\partial{x_2}}=\ \ \frac{\partial{u_2}}{\partial{x_1}},
       \quad \frac{\partial{u_1}}{\partial{x_3}}=\ \ \frac{\partial{u_3}}{\partial{x_1}},
       \quad \frac{\partial{u_2}}{\partial{x_3}}=\ \ \frac{\partial{u_3}}{\partial{x_2}}.
\end{cases}
\label{Bryukhov-hyperbolic-4}
\end{gather}

 New properties of analytic models of potential vector fields may be investigated in more detail in case ${\phi}(x_3) = x_3^{-\alpha}$ $(x_3>0$, $\alpha \in \BR)$.
 The static system~\eqref{conservative-system-4} is expressed as
\begin{gather}
\begin{cases}
 x_3 \mathrm{div}\ { \vec V} - \alpha V_3 =0,  \\[1ex]
    \frac{\partial{V_0}}{\partial{x_1}}= \frac{\partial{V_1}}{\partial{x_0}},
       \quad \frac{\partial{V_0}}{\partial{x_2}}= \frac{\partial{V_2}}{\partial{x_0}},
       \quad \frac{\partial{V_0}}{\partial{x_3}}= \frac{\partial{V_3}}{\partial{x_0}}, \\[1ex]
      \frac{\partial{V_1}}{\partial{x_2}}= \frac{\partial{V_2}}{\partial{x_1}},
       \quad \frac{\partial{V_1}}{\partial{x_3}}= \frac{\partial{V_3}}{\partial{x_1}},
       \quad \frac{\partial{V_2}}{\partial{x_3}}= \frac{\partial{V_3}}{\partial{x_2}},
\end{cases}
\label{alpha-hyperbolic-isotropic-system-4}
\end{gather}
and the system~\eqref{Bryukhov-hyperbolic-4} is simplified:
\begin{gather}
\begin{cases}
 x_3 \left( \frac{\partial{u_0}}{\partial{x_0}}-
      \frac{\partial{u_1}}{\partial{x_1}}-\frac{\partial{u_2}}{\partial{x_2}}-
      \frac{\partial{u_3}}{\partial{x_3}} \right) + \alpha u_3 = 0 \\[1ex]
      \frac{\partial{u_0}}{\partial{x_1}}=-\frac{\partial{u_1}}{\partial{x_0}},
       \quad \frac{\partial{u_0}}{\partial{x_2}}=-\frac{\partial{u_2}}{\partial{x_0}},
       \quad \frac{\partial{u_0}}{\partial{x_3}}=-\frac{\partial{u_3}}{\partial{x_0}}, \\[1ex]
      \frac{\partial{u_1}}{\partial{x_2}}=\ \ \frac{\partial{u_2}}{\partial{x_1}},
       \quad \frac{\partial{u_1}}{\partial{x_3}}=\ \ \frac{\partial{u_3}}{\partial{x_1}},
       \quad \frac{\partial{u_2}}{\partial{x_3}}=\ \ \frac{\partial{u_3}}{\partial{x_2}}.
\end{cases}
\label{H_4^alpha-system}
\end{gather}
The system~\eqref{H_4^alpha-system} was first introduced by Eriksson
and Orelma in 2019 in the context of the theory of \emph{Hyperbolic
function theory in the skew-field of quaternions}
\cite{ErOrel:2019}. This system demonstrates explicitly a family of
generalizations of the Cauchy-Riemann system in accordance with the
system~\eqref{alpha-hyperbolic-isotropic-system-4} for different
values of the parameter $\alpha$.

When $\alpha> 0$, the system~\eqref{H_4^alpha-system} may be
characterized as $\alpha$-hyperbolic non-Euclidean modification of
the system $(R)$ with respect to the conformal metric defined on the
halfspace $\{x_3 > 0\}$ by the formula
\begin{gather*}
ds^2 = \frac{d{x_0}^2 + d{x_1}^2 + d{x_2}^2  +
d{x_3}^2}{x_2^{\alpha}}.
\end{gather*}

The continuity equation~\eqref{Liouville-eq-4-expanded}  takes the
form of the Weinstein equation in $\mathbb R^4$ for any value of the
parameter $\alpha$ (see, e.g.,
\cite{Leut:CV17,HempLeut:1996,ErOrel:2019,Leut:2018-AACA,Leut:2020-CAOT})
\begin{equation}
 x_3 \Delta{h} - \alpha \frac{\partial{h}}{\partial{x_3}} =0.
\label{alpha-hyperbolic-4}
\end{equation}

Meanwhile, nowadays solutions of the Weinstein
equation~\eqref{alpha-hyperbolic-4} in case $\alpha =2$ in the
context of \emph{Hyperbolic function theory in the skew-field of
quaternions} are referred to as 2-hyperbolic harmonic functions in
$\mathbb R^4$ \cite{ErOrel:2019}. The critical sets of 2-hyperbolic
harmonic functions $h= h(x_0,x_1,x_2,x_3)$ within Fueter's
construction in $\mathbb R^4$ \eqref{Fueter-4}, where $F = F(x) =
\frac{\partial{h}}{\partial{x_0}} - i
\frac{\partial{h}}{\partial{x_1}} - j
\frac{\partial{h}}{\partial{x_1}} - k
\frac{\partial{h}}{\partial{x_1}}$,
 under conditions of $x_1 \neq 0, x_2 \neq 0, x_3 \neq 0$ coincide with the sets of zeros
 of $F = F(x)$ (on the structure of the sets of zeros of quaternionic polynomials with real coefficients see, e.g., \cite{PogoruiShapiro:2004,Topuridze:2003}).

\begin{defn}
Let $\Lambda \subset \mathbb R^4$ $ (x_3 > 0)$ be a simply connected
open domain,  $\alpha> 0$. Every exact solution of
Eq.~\eqref{alpha-hyperbolic-4} in $\mathbb R^4$ is called
$\alpha$-hyperbolic harmonic potential in $\mathbb R^4$.
\end{defn}

When $\alpha = 0$, the system~\eqref{H_4^alpha-system} becomes the
system $(R_4)$, while the Weinstein equation in $\mathbb
R^4$~\eqref{alpha-hyperbolic-4} becomes the Laplace equation.

When $\alpha < 0$, solutions of Eq.~\eqref{alpha-hyperbolic-4} in
the context of \emph{Modified harmonic functions in $\mathbb R^4$}
are referred to as $-\alpha$-modified harmonic functions in $\mathbb
R^4$ (see, e.g., \cite{Leut:2018-AACA,Leut:2020-CAOT}). Properties
of homogeneous polynomial solutions of
Eq.~\eqref{alpha-hyperbolic-4} in spherical coordinates on the unit
half-sphere $S_{+}^3 = \{ (x_0, x_1, x_2, x_3): x_0^2 + x_1^2 +
x_2^2 + x_3^2 =1$, $ \ x_3 > 0 \} $ in $\mathbb R^4$ have been
recently studied by Leutwiler.

To provide the second type of potential vector fields in $\mathbb
R^4$, assume that $C^1$-coefficient $\phi(x_0,x_1,x_2,x_3)$ depends
only on the cylindrical radial variable $\rho$ so that $\phi=
\phi(\rho)
> 0$. The system~\eqref{Bryukhov-general-4} is described as
\begin{gather}
\begin{cases}
  \phi(\rho) \left( \frac{\partial{u_0}}{\partial{x_0}}-
      \frac{\partial{u_1}}{\partial{x_1}}- \frac{\partial{u_2}}{\partial{x_2}} - \frac{\partial{u_3}}{\partial{x_3}} \right)
      - \left( \frac{\partial{ \phi(\rho) }}{\partial{x_1}}u_1 + \frac{\partial{ \phi(\rho) }}{\partial{x_2}}u_2 + \frac{\partial{ \phi(\rho) }}{\partial{x_3}}u_3 \right) = 0,  \\[1ex]
       \frac{\partial{u_0}}{\partial{x_1}}=-\frac{\partial{u_1}}{\partial{x_0}},
       \quad  \frac{\partial{u_0}}{\partial{x_2}}=-\frac{\partial{u_2}}{\partial{x_0}},
       \quad  \frac{\partial{u_0}}{\partial{x_3}}=-\frac{\partial{u_3}}{\partial{x_0}}, \\[1ex]
      \frac{\partial{u_1}}{\partial{x_2}}=\ \ \frac{\partial{u_2}}{\partial{x_1}},
       \quad  \frac{\partial{u_1}}{\partial{x_3}}=\ \ \frac{\partial{u_3}}{\partial{x_1}},
       \quad \frac{\partial{u_2}}{\partial{x_3}}=\ \
       \frac{\partial{u_3}}{\partial{x_2}}.
\end{cases}
\label{Bryukhov-axial-4}
\end{gather}

New properties of axially symmetric analytic models of potential
vector fields may be investigated in more detail in case $\phi(
\rho) = \rho^{-\alpha}$ $ (\rho > 0$, $\alpha \in \BR)$. The static
system~\eqref{conservative-system-4} is expressed as
\begin{gather}
\begin{cases}
  (x_1^2+x_2^2+x_3^2) \ \mathrm{div}\ { \vec V} -
 \alpha \left(x_1 V_1 + x_2 V_2 + x_3 V_3 \right) =0,  \\[1ex]
    \frac{\partial{V_0}}{\partial{x_1}}= \frac{\partial{V_1}}{\partial{x_0}},
       \quad \frac{\partial{V_0}}{\partial{x_2}}= \frac{\partial{V_2}}{\partial{x_0}},
       \quad \frac{\partial{V_0}}{\partial{x_3}}= \frac{\partial{V_3}}{\partial{x_0}}, \\[1ex]
      \frac{\partial{V_1}}{\partial{x_2}}= \frac{\partial{V_2}}{\partial{x_1}},
       \quad \frac{\partial{V_1}}{\partial{x_3}}= \frac{\partial{V_3}}{\partial{x_1}},
       \quad \frac{\partial{V_2}}{\partial{x_3}}= \frac{\partial{V_3}}{\partial{x_2}},
\end{cases}
\label{alpha-axial-isotropic-system-4}
\end{gather}
and the system~\eqref{Bryukhov-axial-4} is simplified:
\begin{gather}
\begin{cases}
 (x_1^2+x_2^2+x_3^2) \left( \frac{\partial{u_0}}{\partial{x_0}}-
      \frac{\partial{u_1}}{\partial{x_1}}-\frac{\partial{u_2}}{\partial{x_2}}-
      \frac{\partial{u_3}}{\partial{x_3}} \right) + \alpha (x_1u_1+x_2u_2+x_3u_3)=0 \\[1ex]
      \frac{\partial{u_0}}{\partial{x_1}}=-\frac{\partial{u_1}}{\partial{x_0}},
       \quad \frac{\partial{u_0}}{\partial{x_2}}=-\frac{\partial{u_2}}{\partial{x_0}},
       \quad \frac{\partial{u_0}}{\partial{x_3}}=-\frac{\partial{u_3}}{\partial{x_0}}, \\[1ex]
      \frac{\partial{u_1}}{\partial{x_2}}=\ \ \frac{\partial{u_2}}{\partial{x_1}},
       \quad \frac{\partial{u_1}}{\partial{x_3}}=\ \ \frac{\partial{u_3}}{\partial{x_1}},
       \quad \frac{\partial{u_2}}{\partial{x_3}}=\ \ \frac{\partial{u_3}}{\partial{x_2}}.
\end{cases}
\label{eq:A_4^alpha-system}
\end{gather}
This system demonstrates explicitly a family of axially symmetric
generalizations of the Cauchy-Riemann system in accordance with the
system~\eqref{alpha-axial-isotropic-system-4} for different values
of the parameter $\alpha$.

The continuity equation~\eqref{Liouville-eq-4-expanded} is written
as
 \begin{equation}
(x_1^2+ x_2^2 + x_3^2)\Delta{h} - \alpha \left(
x_1\frac{\partial{h}}{\partial{x_1}} +
x_2\frac{\partial{h}}{\partial{x_2}} +
x_3\frac{\partial{h}}{\partial{x_3}} \right)  =0.
 \label{alpha-axial-hyperbolic-4}
  \end{equation}

When $\alpha > 0$, the system~\eqref{eq:A_4^alpha-system} may be
characterized as $\alpha$-axial-hy\-per\-bo\-lic non-Euclidean
modification of the system $(R_4)$ with respect to the conformal
metric defined outside the axis $x_0$ by the formula
$$
ds^2 = \frac{d{x_0}^2 + d{x_1}^2 + d{x_2}^2 +
d{x_3}^2}{\rho^{\alpha}}.
$$
\begin{defn}
Let $\Lambda \subset \mathbb R^4$ $(\rho > 0)$ be a simply connected
open domain, $\alpha > 0$.  Every exact  solution of
Eq.~\eqref{alpha-axial-hyperbolic-4} in $\mathbb R^4$ is called
$\alpha$-axial-hyperbolic harmonic potential in $\mathbb R^4$.
\end{defn}

Let us compare properties of $\alpha$-hyperbolic harmonic potentials
and $\alpha$-axial-hyperbolic harmonic potentials in $\mathbb R^4$
in Cartesian coordinates. This immediately leads to the following
formulation.

 \begin{prop} [The first criterion]
  Any $\alpha$-hyperbolic harmonic potential $h= h(x_0, x_1, x_2, x_3)$ in  $\Lambda \subset \mathbb R^4$ $(x_3 > 0)$
 represents an $\alpha$-axial-hyperbolic harmonic potential in $\Lambda$ if and only if
$ x_2 \frac{\partial{h}}{\partial{x_1}} = x_1
\frac{\partial{h}}{\partial{x_2}},$
$x_3\frac{\partial{h}}{\partial{x_1}} = x_1
\frac{\partial{h}}{\partial{x_3}},$
$x_3\frac{\partial{h}}{\partial{x_2}} = x_2
\frac{\partial{h}}{\partial{x_3}}.$
 Each condition $ x_m=0$ $(m=1,2,3)$ within joint class of $\alpha$-hyperbolic harmonic and $\alpha$-axial-hyperbolic harmonic potentials
  implies that the component $u_m= \frac{\partial{h}}{\partial{x_m}}$ vanishes.
 \end{prop}

\begin{rem}
 Necessary and sufficient conditions of joint class of $\alpha$-hyperbolic harmonic and
 $\alpha$-axial-hyperbolic harmonic potentials in
$\mathbb R^4$ coincide with conditions~\eqref{spec.cond-4}
 of joint class of analytic solutions of the system $(H_4)$ and the
system~\eqref{Bryukhov-4}.
\end{rem}

Let us now compare properties of $\alpha$-hyperbolic harmonic
potentials and $\alpha$-axial-hyperbolic harmonic potentials in
$\mathbb R^4$ in cylindrical coordinates.
Eq.~\eqref{alpha-axial-hyperbolic-4} in cylindrical coordinates is
written as
 \begin{gather*}
  \rho^2 \left( \frac{\partial{^2}{h}}{\partial{x_0}^2} +  \frac{\partial {^2}{h}}{\partial{\rho}^2} \right)
  - (\alpha - 2) \rho \frac{\partial{h}}{\partial{\rho}} +
 \cot{\theta} \frac{\partial{h}}{\partial{\theta}} +
         \frac{\partial{^2}{h}}{\partial{\theta}^2}+
        \frac{1}{ \sin^2 \theta} \frac{\partial {^2}{h}}{\partial{\psi}^2}  = 0.
\end{gather*}

The Weinstein equation in $\mathbb R^4$ $(\ref{alpha-hyperbolic-4})$
in cylindrical coordinates takes the following form:
 \begin{gather*}
 \rho^2 \left( \frac{\partial{^2}{h}}{\partial{x_0}^2} +  \frac{\partial {^2}{h}}{\partial{\rho}^2} \right)
          - (\alpha - 2) \rho \frac{\partial{h}}{\partial{\rho}}
 - ( \alpha -1) \cot{\theta} \frac{\partial{h}}{\partial{\theta}}  +
         \frac{\partial{^2}{h}}{\partial{\theta}^2}+
             \frac{1}{ \sin^2 \theta} \frac{\partial {^2}{h}}{\partial{\psi}^2} \\[1ex]
- \alpha \frac{\cot{\psi}}{ \sin^2{\theta}}
\frac{\partial{h}}{\partial{\psi}}
 = 0.
 \end{gather*}

This immediately leads to the following formulation.
 \begin{prop} [The second criterion]
  Every $\alpha$-hyperbolic harmonic potential $h= h(x_0, x_1, x_2, x_3)$ in  $\Lambda \subset \mathbb R^4$ $(x_3 > 0)$
 represents an $\alpha$-axial-hyperbolic harmonic potential in $\Lambda$ if and only if  in cylindrical coordinates
$ \frac{\partial{h}}{\partial{\theta}} = 0,
 \frac{\partial{h}}{\partial{\psi}} = 0.$
\end{prop}

\begin{rem}
New approach may be efficiently developed in the context of the
theory of \emph{Hyperbolic function theory in the skew-field of
quaternions} and the theory of \emph{Modified harmonic functions in
$\mathbb R^4$} under conditions of
$\frac{\partial{h}}{\partial{\theta}} = 0$,
$\frac{\partial{h}}{\partial{\psi}} = 0$.
\end{rem}

It should be noted that the vector $\vec V= (V_0, V_1, V_2, V_3)$
within potential meridional fields in  $\mathbb R^4$ is independent
of two angles $\psi$ and $\theta$, herewith $V_{\psi} :=
\frac{\partial{h}}{\partial{\psi}} \equiv 0$ and $V_{\theta} :=
\frac{\partial{h}}{\partial{\theta}} \equiv 0$.

 As it follows from the first and second criterions,
 new joint class of $\alpha$-hyperbolic harmonic and $\alpha$-axial-hyperbolic harmonic potentials
 in  simply connected open domains $\Lambda \subset
\mathbb R^4$ $(x_3 > 0)$ may be characterized as general class of
potential meridional fields $ \vec V$, where $\phi( \rho) =
\rho^{-\alpha}$. Every scalar potential $h$ within the joint class
is independent of angles $\psi$, $\theta$ such that $h(x_0, \rho,
\theta, \psi)$ $ := g(x_0, \rho),$ $V_0 =
\frac{\partial{g}}{\partial{x_0}},$
 $V_1 = \frac{\partial{g}}{\partial{\rho}} \frac{x_1}{\rho}$,
$V_2 = \frac{\partial{g}}{\partial{\rho}} \frac{x_2}{\rho}$, $V_3 =
\frac{\partial{g}}{\partial{\rho}} \frac{x_3}{\rho}$, herewith
\begin{gather}
 \rho \left( \frac{\partial{^2}{g}}{\partial{x_0}^2} +  \frac{\partial {^2}{g}}{\partial{\rho}^2} \right)
  - (\alpha -2) \frac{\partial{g}}{\partial{\rho}}
  = 0.
  \label{Euler-Poisson-Darboux equation-alpha}
  \end{gather}
Equation~\eqref{Euler-Poisson-Darboux equation-alpha}, where
$\hat{k}= - (\alpha -2)$, is referred to as the elliptic
Euler-Poisson-Darboux type equation in cylindrical coordinates (see,
e.g., \cite{Dzhaiani,Aksenov:2005,Br:Hefei2020}), or generalized
axially symmetric potential equation \emph{(GASPE)} in the context
of \emph{GASPT} (see, e.g., \cite{Weinstein:1953,Huber:1954,Colton,
Zwillinger,GrPlaksa:2009}). Exact solutions $g = g(x_0, \rho)$ of
Eq.~\eqref{Euler-Poisson-Darboux equation-alpha} are often referred
to as generalized axially symmetric potentials.

Every exact solution of Eq.~\eqref{Euler-Poisson-Darboux
equation-alpha} as generalized axially symmetric potential indicates
the existence of the Stokes' stream function $\hat{g} = \hat{g}(x_0,
\rho)$, which is defined by the generalized Stokes-Beltrami system
in the meridian half-plane $(\rho
> 0)$
 (see, e.g., \cite{Weinstein:1953,Br:Hefei2020}):
\begin{equation}
  \left\{
       \begin{array}{l}
      {\rho}^{2 - \alpha} \frac{\partial{g}}{\partial{x_0}} = \frac{\partial{\hat{g}}}{\partial{\rho}},  \\
       {\rho}^{2 - \alpha} \frac{\partial{g}}{\partial{\rho}}=-\frac{\partial{\hat{g}}}{\partial{x_0}}.
      \end{array}
   \right.
\label{generalized Stokes-Beltrami}
\end{equation}
 The Stokes' stream function $\hat{g} = \hat{g}(x_0, \rho)$, in contrast to generalized axially symmetric potential $g = g(x_0, \rho)$, satisfies the following equation:
$$
  \rho \left( \frac{\partial{^2}{\hat{g}}}{\partial{x_0}^2} +  \frac{\partial {^2}{\hat{g}}}{\partial{\rho}^2} \right)
  + (\alpha -2) \frac{\partial{\hat{g}}}{\partial{\rho}} = 0.
$$

Consider a special class of solutions of Eq.
$(\ref{Euler-Poisson-Darboux equation-alpha})$ under condition of
separation of variables $g(x_0,  \rho) = \Xi(x_0)  \Upsilon(\rho)$:
$$
 \frac{1}{\Xi}  \frac{d{^2}{\Xi}}{d{x_0}^2} +   \frac{1}{ \Upsilon} \frac{d{^2}{ \Upsilon}}{d{\rho}^2}
 - \frac{(\alpha -2)} { \Upsilon \rho} \frac{d{ \Upsilon}}{d{\rho}} = 0.
$$
 Relations
\begin{gather*}
  -   \frac{1}{\Xi} \frac{d{^2}{\Xi}}{d{x_0}^2} =
    \frac{1}{ \Upsilon} \frac{d{^2}{ \Upsilon}}{d{\rho}^2}
 - \frac{(\alpha -2)} { \Upsilon \rho} \frac{d{ \Upsilon}}{d{\rho}} =
     - \breve{\beta}^2  \ \ \ \ \  ( \breve{\beta}  = const \in  \mathbf R )
  \end{gather*}
are equivalent to the following system of ordinary differential
equations:
\begin{equation}
\left\{
      \begin{array}{l}
    \frac{d{^2}{\Xi}}{d{x_0}^2} - \breve{\beta}^2  \Xi = 0, \\
 \rho^2 \frac{d{^2}{ \Upsilon}}{d{\rho}^2}
 - (\alpha -2) \rho \frac{d{ \Upsilon}}{d{\rho}}
  + \breve{\beta}^2 \rho^2 \Upsilon = 0.
     \end{array}
  \right.
  \label{eq-sep-x_2-x_0-hyper-cyl}
  \end{equation}
The first equation of the system $(\ref{eq-sep-x_2-x_0-hyper-cyl})$
may be solved using hyperbolic functions: $\
\Xi_{\breve{\beta}}(x_0) = b^1_{\breve{\beta}} \cosh{\breve{\beta}
x_0} +  b^2_{\breve{\beta}} \sinh{\breve{\beta} x_0}$;
 $\  b^1_{\breve{\beta}},  b^2_{\breve{\beta}}= const \in \mathbf R$.
 In particular, values $ b^1_{\breve{\beta}} = b^2_{\breve{\beta}}= 1$ imply that $ \
\Xi_{\breve{\beta}}(x_0) = e^{\breve{\beta} x_0}$  (see, e.g.,
\cite{BrKaeh:2016}).

The second equation of the system $(\ref{eq-sep-x_2-x_0-hyper-cyl})$
may be solved using linear independent solutions

$ \Upsilon_{\breve{\beta}}(\rho) = {\rho}^\frac{\alpha -1}{2} \left[
a^1_{\breve{\beta}} J_{\frac{\alpha -1}{2}}( \breve{\beta} \rho) +
a^2_{\breve{\beta}} Y_{\frac{\alpha -1}{2}}( \breve{\beta} \rho)
\right]$;
 $ \  a^1_{\breve{\beta}}$, $ a^2_{\breve{\beta}}= const \in \mathbf R$, \\
where $J_{\frac{\alpha -1}{2}}( \breve{\beta} \rho)$ and
$Y_{\frac{\alpha -1}{2}}( \breve{\beta} \rho)$ are Bessel functions
of the first and second kind of order $\frac{\alpha -1}{2}$ and real
argument $\breve{\beta} \rho$ (see, e.g.,
\cite{Watson:1944,BatEr-Higher-II,PolZait:Ordin-2017}).

\section{Potential Meridional Fields in $\mathbb R^4$ and Four-Dimensional $\alpha$-Meridional Mappings of the Second Kind}

Equation~\eqref{Euler-Poisson-Darboux equation-alpha} in cylindrical
coordinates within potential meridional fields in $\mathbb R^4$
leads to a family of Vekua type systems for different values of the
parameter $\alpha$ investigated by Sommen et al. in the context of
the theory of \emph{Quaternion-valued monogenic functions of axial
type} (see, e.g.,
\cite{PenaSommen:2012,PenaSabSommen:2017,ErOrelVie:2017})
\begin{gather}
\begin{cases}
\rho \left( \frac{\partial{u_0}}{\partial{x_0}} - \frac{\partial{u_{\rho}}}{\partial{\rho}} \right)  +  (\alpha -2) u_{\rho} = 0,\\[1ex]
\frac{\partial{u_0}}{\partial{\rho}}=-\frac{\partial{u_{\rho}}}{\partial{x_0}}.
\end{cases}
\label{A_4^alpha system-meridional}
\end{gather}
We should take into account that in our setting $ u_0 =
\frac{\partial{g}}{\partial{x_0}}, \ u_{\rho} = -
\frac{\partial{g}}{\partial{\rho}}.$

The static system~\eqref{alpha-axial-isotropic-system-4} is reduced
to the following two-di\-men\-sio\-nal system in the meridian
half-plane:
\begin{gather}
\begin{cases}
 \rho \left( \frac{\partial{V_0}}{\partial{x_0}} + \frac{\partial{V_{\rho}}}{\partial{\rho}} \right)  -  (\alpha -2) V_{\rho} = 0, \\[1ex]
      \frac{\partial{V_0}}{\partial{\rho}} = \frac{\partial{V_{\rho}}}{\partial{x_0}},
\end{cases}
\label{Bryukhov-merid-4}
\end{gather}
where $V_0= u_0$, $V_{\rho} = -u_{\rho}$, and $V_1  =
V_{\rho}\frac{x_1}{\rho}$, $V_2  = V_{\rho}\frac{x_2}{\rho}$, $V_3 =
V_{\rho}\frac{x_3}{\rho}$.

The principal invariants of the Jacobian matrix $\mathbf{J}(\vec
V(x))$ may be demonstrated explicitly. It should be noted that the
Jacobian matrix~\eqref{Hessian-matrix-4} is substantially
simplified:
\begin{gather}
\tiny{
\begin{pmatrix}
 \left[ -\frac{\partial{V_{\rho}}}{\partial{\rho}}
+ \frac{V_{\rho}}{\rho}(\alpha -2) \right]  &
\frac{\partial{V_{\rho}}}{\partial{x_0}} \frac{x_1}{\rho} &
 \frac{\partial{V_{\rho}}}{\partial{x_0}} \frac{x_2}{\rho} &  \frac{\partial{V_{\rho}}}{\partial{x_0}} \frac{x_3}{\rho}  \\[1ex]
\frac{\partial{V_{\rho}}}{\partial{x_0}} \frac{x_1}{\rho}  & \left(
\frac{\partial{V_{\rho}}}{\partial{\rho}} \frac{x_1^2}{\rho^2}  +
\frac{V_{\rho}}{\rho} \frac{x_2^2+x_3^2}{\rho^2}\right)  &
 \left( \frac{\partial{V_{\rho}}}{\partial{\rho}}- \frac{V_{\rho}}{\rho}\right)  \frac{x_1 x_2}{\rho^2} &
 \left( \frac{\partial{V_{\rho}}}{\partial{\rho}}- \frac{V_{\rho}}{\rho}\right)  \frac{x_1 x_3}{\rho^2}  \\[1ex]
          \frac{\partial{V_{\rho}}}{\partial{x_0}} \frac{x_2}{\rho}  & \left(
\frac{\partial{V_{\rho}}}{\partial{\rho}}-
\frac{V_{\rho}}{\rho}\right)  \frac{x_1 x_2}{\rho^2}  & \left(
\frac{\partial{V_{\rho}}}{\partial{\rho}} \frac{x_2^2}{\rho^2} +
\frac{V_{\rho}}{\rho} \frac{x_1^2+x_3^2}{\rho^2}\right) & \left(
\frac{\partial{V_{\rho}}}{\partial{\rho}}-
\frac{V_{\rho}}{\rho}\right)  \frac{x_2 x_3}{\rho^2}     \\[1ex]
          \frac{\partial{V_{\rho}}}{\partial{x_0}} \frac{x_3}{\rho}  & \left(
\frac{\partial{V_{\rho}}}{\partial{\rho}}-
\frac{V_{\rho}}{\rho}\right)  \frac{x_1 x_3}{\rho^2}   & \left(
\frac{\partial{V_{\rho}}}{\partial{\rho}}-
\frac{V_{\rho}}{\rho}\right)  \frac{x_2 x_3}{\rho^2} & \left(
\frac{\partial{V_{\rho}}}{\partial{\rho}} \frac{x_3^2}{\rho^2} +
\frac{V_{\rho}}{\rho} \frac{x_1^2+x_2^2}{\rho^2}\right)
 \label{Jacobian-merid-4}
\end{pmatrix}
}
\end{gather}

\begin{thm}
Roots of the characteristic equation~\eqref{characteristic lambda-4}
of the Jacobian matrix~\eqref{Jacobian-merid-4} are given by exact
formulas:
\begin{align*}
\lambda_{0,1}
&= \frac{V_{\rho}}{\rho},\\[1ex]
\lambda_{2,3}
&=\frac{(\alpha -2)}{2} \frac{ V_{\rho}}{ \rho}  \pm \\[1ex]
&\hspace*{5ex}\sqrt{ \frac{(\alpha -2)^2}{4} \left( \frac{V_{\rho}}{
\rho} \right)^2 - (\alpha -2) \frac{ V_{\rho}}{\rho}
\frac{\partial{V_{\rho}}}{\partial{\rho}}+ \left(
\frac{\partial{V_{\rho}}}{\partial{x_0}}\right)^2 + \left(
\frac{\partial{V_{\rho}}}{\partial{\rho}} \right)^2}.
\end{align*}
\end{thm}
\begin{proof}
The principal invariants of the Jacobian
matrix~\eqref{Jacobian-merid-4} are written as
\begin{align*}
I_{\mathbf{J}(\vec V)}
&= \alpha \frac{V_{\rho}}{\rho}, \\[1ex]
II_{\mathbf{J}(\vec V)} &= - \left[ \left(
\frac{\partial{V_\rho}}{\partial{x_0}}  \right)^2 +  \left(
\frac{\partial{V_{\rho}}}{\partial{\rho}} \right)^2 \right] +
(\alpha -2) \frac{V_{\rho}}{\rho}
\frac{\partial{V_{\rho}}}{\partial{\rho}} + (2\alpha -3) \left(
\frac{V_{\rho}}{ \rho} \right)^2, \\[1ex]
III_{\mathbf{J}(\vec V)} &= -2 \frac{V_{\rho}}{\rho} \left[  \left(
\frac{\partial{V_\rho}}{\partial{x_0}}  \right)^2 +  \left(
\frac{\partial{V_{\rho}}}{\partial{\rho}} \right)^2 \right] +
(\alpha -2) \left( \frac{V_{\rho}}{ \rho} \right)^2
  \left(2 \frac{\partial{V_{\rho}}}{\partial{\rho}} +
\frac{V_{\rho}}{ \rho} \right), \\[1ex]
IV_{\mathbf{J}(\vec V)} &= -\left( \frac{V_{\rho}}{ \rho} \right)^2
\left[ \left( \frac{\partial{V_\rho}}{\partial{x_0}} \right)^2 +
\left( \frac{\partial{V_{\rho}}}{\partial{\rho}} \right)^2 \right] +
(\alpha -2) \left( \frac{V_{\rho}}{ \rho} \right)^3
\frac{\partial{V_{\rho}}}{\partial{\rho}}.
\end{align*}
The characteristic equation~\eqref{characteristic lambda-4} into the
framework of the system~\eqref{Bryukhov-merid-4} may be factored:
\begin{align*}
&\left( \lambda - \frac{V_{\rho}}{\rho} \right)^2\times\\
&\left[ \lambda^2 - (\alpha -2) \frac{ V_{\rho}}{\rho}\lambda +
(\alpha -2) \frac{ V_{\rho}}{\rho}
\frac{\partial{V_{\rho}}}{\partial{\rho}} - \left(
\frac{\partial{V_{\rho}}}{\partial{x_0}}\right)^2 - \left(
\frac{\partial{V_{\rho}}}{\partial{\rho}} \right)^2 \right] = 0.
\qedhere
\end{align*}
\end{proof}
\begin{cor}[On the set of degenerate points]
Assume that a potential meridional field $\vec V = (V_0,
\frac{x_1}{\rho} V_{\rho}, \frac{x_2}{\rho} V_{\rho},
\frac{x_3}{\rho} V_{\rho})$ satisfies the
system~\eqref{Bryukhov-merid-4}. The set of degenerate points of the
Jacobian matrix~\eqref{Jacobian-merid-4} is provided by two
independent equations:
$$
{V_{\rho}}=0,\quad
\left(\frac{\partial{V_{\rho}}}{\partial{x_0}}\right)^2 +
\left(\frac{\partial{V_{\rho}}}{\partial{\rho}}\right)^2 -
(\alpha-2)\frac{V_{\rho}}{\rho}\frac{\partial{V_{\rho}}}{\partial{\rho}}=0.
$$
\end{cor}
\begin{cor}[On the zero divergence condition]
Assume that a potential meridional field $\vec V = (V_0,
\frac{x_1}{\rho} V_{\rho}, \frac{x_2}{\rho} V_{\rho},
\frac{x_3}{\rho} V_{\rho})$ satisfies the
system~\eqref{Bryukhov-merid-4}, where $\alpha \neq 0$. Every point
$x = (x_0, x_1, x_2, x_3)$, where $\mathrm{div} \, \vec V (x_0, x_1,
x_2, x_3) = 0$, is a degenerate point of the Jacobian
matrix~\eqref{Jacobian-merid-4}.
\end{cor}
\begin{rem}
The second formula of \textbf{Theorem 3.1} may be simplified:
\begin{align*}
 \lambda_{2,3}
&= \frac{(\alpha -2)}{2} \frac{V_{\rho}}{\rho}  \pm
 \sqrt{\left( \frac{\alpha -2}{2}
\frac{V_{\rho}}{\rho} - \frac{\partial{V_{\rho}}}{\partial{\rho}}
\right)^2 + \left(
\frac{\partial{V_{\rho}}}{\partial{x_0}}\right)^2}.
\end{align*}
\end{rem}

Geometric properties of the the Jacobian
matrix~\eqref{Jacobian-merid-4} allow us to introduce the concept of
four-dimensional $\alpha$-meridional mappings of the first and
second kind.
\begin{defn}
Let $\alpha$ be a real parameter, while $\Lambda \subset \mathbb
R^4$ be a simply connected open domain, where $x_1 \neq 0, x_2 \neq
0, x_3 \neq 0$. Assume that an exact solution $(u_0, u_1, u_2, u_3)$
of the system~\eqref{eq:A_4^alpha-system},  where $\alpha \neq 0$,
satisfies axially symmetric conditions ${u_1}{x_2}={u_2}{x_1}, \quad
{u_1}{x_3}={u_3}{x_1}, \quad {u_2}{x_3}={u_3}{x_2}$ in $\Lambda$.
Then mapping $u = u_0 + iu_1 + ju_2 + ku_3: \Lambda \rightarrow
\mathbb{R}^4$ is called four-dimensional $\alpha$-meridional mapping
of the first kind, while mapping $ \overline{u} = u_0 - iu_1 - ju_2
- ku_3: \Lambda \rightarrow \mathbb{R}^4$ is called four-dimensional
$\alpha$-meridional mapping of the second kind.
\end{defn}

 The set of degenerate points of every four-dimensional $\alpha$-meridional mapping
of the second kind coincides with the set of degenerate points of
the Jacobian matrix~\eqref{Jacobian-merid-4}.

\begin{rem}
The concept of four-dimensional $\alpha$-meridional mappings of the
first and second kind may be efficiently developed in the context of
the theory of \emph{Hyperbolic function theory in the skew-field of
quaternions} and the theory of \emph{Modified harmonic functions in
$\mathbb R^4$} under conditions of
$\frac{\partial{h}}{\partial{\theta}} = 0$,
$\frac{\partial{h}}{\partial{\psi}} = 0$.
\end{rem}

Let us first look at the basic properties of potential meridional
fields in the context of the theory of \emph{Modified harmonic
functions in $\mathbb R^4$} in case $\alpha = -2$, where the
systems~\eqref{A_4^alpha
system-meridional},~\eqref{Bryukhov-merid-4} are expressed as
\begin{gather*}
\begin{cases}
         \rho \left( \frac{\partial{u_0}}{\partial{x_0}} - \frac{\partial{u_{\rho}}}{\partial{\rho}} \right) - 4 u_{\rho} = 0,  \\
      \frac{\partial{u_0}}{\partial{\rho}}=-\frac{\partial{u_{\rho}}}{\partial{x_0}},
\end{cases}
\end{gather*}
\begin{gather*}
\begin{cases}
         \rho \left( \frac{\partial{V_0}}{\partial{x_0}} + \frac{\partial{V_{\rho}}}{\partial{\rho}} \right) + 4 V_{\rho} = 0,  \\
      \frac{\partial{V_0}}{\partial{\rho}} =
      \frac{\partial{V_{\rho}}}{\partial{x_0}},
\end{cases}
\end{gather*}
\begin{gather*}
\rho \left( \frac{\partial{^2}{g}}{\partial{x_0}^2} + \frac{\partial
{^2}{g}}{\partial{\rho}^2} \right) + 4
\frac{\partial{g}}{\partial{\rho}} = 0, \quad \quad \rho \left(
\frac{\partial{^2}{\hat{g}}}{\partial{x_0}^2} + \frac{\partial
{^2}{\hat{g}}}{\partial{\rho}^2} \right) - 4
\frac{\partial{\hat{g}}}{\partial{\rho}} = 0.
\end{gather*}

The Jacobian matrix~\eqref{Jacobian-merid-4} takes the following
form:
\begin{gather}
\small{
\begin{pmatrix}
 \left[ -\frac{\partial{V_{\rho}}}{\partial{\rho}}
- 4\frac{V_{\rho}}{\rho} \right]  &
\frac{\partial{V_{\rho}}}{\partial{x_0}} \frac{x_1}{\rho} &
 \frac{\partial{V_{\rho}}}{\partial{x_0}} \frac{x_2}{\rho} &  \frac{\partial{V_{\rho}}}{\partial{x_0}} \frac{x_3}{\rho}  \\[1ex]
\frac{\partial{V_{\rho}}}{\partial{x_0}} \frac{x_1}{\rho}  & \left(
\frac{\partial{V_{\rho}}}{\partial{\rho}} \frac{x_1^2}{\rho^2}  +
\frac{V_{\rho}}{\rho} \frac{x_2^2+x_3^2}{\rho^2}\right)  &
 \left( \frac{\partial{V_{\rho}}}{\partial{\rho}}- \frac{V_{\rho}}{\rho}\right)  \frac{x_1 x_2}{\rho^2} &
 \left( \frac{\partial{V_{\rho}}}{\partial{\rho}}- \frac{V_{\rho}}{\rho}\right)  \frac{x_1 x_3}{\rho^2}  \\[1ex]
          \frac{\partial{V_{\rho}}}{\partial{x_0}} \frac{x_2}{\rho}  & \left(
\frac{\partial{V_{\rho}}}{\partial{\rho}}-
\frac{V_{\rho}}{\rho}\right)  \frac{x_1 x_2}{\rho^2}  & \left(
\frac{\partial{V_{\rho}}}{\partial{\rho}} \frac{x_2^2}{\rho^2} +
\frac{V_{\rho}}{\rho} \frac{x_1^2+x_3^2}{\rho^2}\right) & \left(
\frac{\partial{V_{\rho}}}{\partial{\rho}}-
\frac{V_{\rho}}{\rho}\right)  \frac{x_2 x_3}{\rho^2}     \\[1ex]
          \frac{\partial{V_{\rho}}}{\partial{x_0}} \frac{x_3}{\rho}  & \left(
\frac{\partial{V_{\rho}}}{\partial{\rho}}-
\frac{V_{\rho}}{\rho}\right)  \frac{x_1 x_3}{\rho^2}   & \left(
\frac{\partial{V_{\rho}}}{\partial{\rho}}-
\frac{V_{\rho}}{\rho}\right)  \frac{x_2 x_3}{\rho^2} & \left(
\frac{\partial{V_{\rho}}}{\partial{\rho}} \frac{x_3^2}{\rho^2} +
\frac{V_{\rho}}{\rho} \frac{x_1^2+x_2^2}{\rho^2}\right)
 \label{Jacobian-modif-harm-merid-4}
\end{pmatrix}
}
\end{gather}

\begin{cor}
Roots of the characteristic equation~\eqref{characteristic lambda-4}
are given by the formulas
\begin{gather}
 \lambda_{0,1} = \frac{V_{\rho}}{\rho}, \; \lambda_{2,3} = - 2\frac{V_{\rho}}{\rho}  \pm
 \sqrt{\left(
2\frac{V_{\rho}}{\rho} + \frac{\partial{V_{\rho}}}{\partial{\rho}}
\right)^2 + \left(
\frac{\partial{V_{\rho}}}{\partial{x_0}}\right)^2}.
 \label{modif-harm-merid-lambda-4}
\end{gather}
\end{cor}

Exact formulas~\eqref{modif-harm-merid-lambda-4} demonstrate
explicitly the geometric specifics of the Jacobian
matrix~\eqref{Jacobian-modif-harm-merid-4}.

 \begin{cor}
Suppose that $\alpha = -2$, $\frac{V_{\rho}}{\rho} = \frac{V_1}{x_1}
= \frac{V_2}{x_2} = \frac{V_3}{x_3}$. The set of degenerate points
of the Jacobian matrix~\eqref{Jacobian-modif-harm-merid-4} is
provided by two independent equations:
$$
{V_{\rho}}=0,\quad
\left(\frac{\partial{V_{\rho}}}{\partial{x_0}}\right)^2
+\left(\frac{\partial{V_{\rho}}}{\partial{\rho}}\right)^2 + 4
\frac{V_{\rho}}{\rho}\frac{\partial{V_{\rho}}}{\partial{\rho}}=0.
 $$
 \end{cor}

\begin{ex}
Consider a generalized axially symmetric potential in case $\alpha =
-2$ using Bessel function of the first kind of order $-\frac{3}{2}$
and real argument $\breve{\beta} \rho$:
 $$
 g(x_0,  \rho) = e^{\breve{\beta} x_0} \rho^{-\frac{3}{2}} J_{-\frac{3}{2}}( \breve{\beta} \rho),
 $$
where $ \rho > 0$.  We deal with analytic models of 2-modified
harmonic meridional fields $\vec V= (V_0,\frac{x_1}{\rho}
V_{\rho},\frac{x_2}{\rho} V_{\rho},\frac{x_3}{\rho} V_{\rho})$,
where
 \begin{align*}
V_0 = \breve{\beta} e^{\breve{\beta} x_0} \rho^{-\frac{3}{2}}
J_{-\frac{3}{2}}(\breve{\beta} \rho), \quad
 V_{\rho} = e^{\breve{\beta} x_0}
 \rho^{-\frac{3}{2}} \left( J'_{-\frac{3}{2}}( \breve{\beta} \rho) - \frac{3}{2 \rho} J_{-\frac{3}{2}}( \breve{\beta} \rho)
 \right),
 \end{align*}
 such that
 \begin{align*}
\frac{\partial{V_{\rho}}}{\partial{x_0}} &= \breve{\beta}
e^{\breve{\beta} x_0} \rho^{-\frac{3}{2}} \left( J'_{-\frac{3}{2}}(
\breve{\beta} \rho) - \frac{3}{2 \rho} J_{-\frac{3}{2}}(
\breve{\beta} \rho) \right),\\
\frac{\partial{V_{\rho}}}{\partial{\rho}} &= e^{\breve{\beta} x_0}
\rho^{-\frac{3}{2}} \left( J''_{-\frac{3}{2}}( \breve{\beta} \rho)-
\frac{3}{\rho} J'_{-\frac{3}{2}}( \breve{\beta} \rho) + \frac{15}{4
\rho^2} J_{-\frac{3}{2}}( \breve{\beta} \rho) \right).
\end{align*}

Roots of the characteristic equation~\eqref{characteristic lambda-4}
of the Jacobian matrix~\eqref{Jacobian-modif-harm-merid-4} are given
by the formulas \small{
\begin{align*}
&\lambda_{0,1} = \frac{e^{\breve{\beta} x_0}}{\rho^2 \sqrt{\rho}}
\left( J'_{-\frac{3}{2}}( \breve{\beta} \rho) - \frac{3}{2 \rho}
J_{-\frac{3}{2}}( \breve{\beta} \rho)
 \right),\\
&\lambda_{2,3}= -2 \frac{e^{\breve{\beta} x_0}}{\rho^2 \sqrt{\rho}}
\left( J'_{-\frac{3}{2}}( \breve{\beta} \rho) - \frac{3}{2 \rho}
J_{-\frac{3}{2}}( \breve{\beta} \rho)
 \right) \pm \\
&\frac{e^{\breve{\beta} x_0}}{\rho \sqrt{\rho}} \sqrt{
\breve{\beta}^2 \left( J'_{-\frac{3}{2}}( \breve{\beta} \rho) -
\frac{3}{2 \rho} J_{-\frac{3}{2}}( \breve{\beta} \rho) \right)^2 +
\left( J''_{-\frac{3}{2}}( \breve{\beta} \rho) - \frac{1}{\rho}
J'_{-\frac{3}{2}}( \breve{\beta} \rho) + \frac{3}{4 \rho^2}
J_{-\frac{3}{2}}( \breve{\beta} \rho) \right)^2}.
\end{align*}
}
 The set of degenerate points of the Jacobian
matrix~\eqref{Jacobian-modif-harm-merid-4} is provided by two
independent equations: \small{
\begin{gather*}
J'_{-\frac{3}{2}}( \breve{\beta} \rho) - \frac{3}{2 \rho} J_{-\frac{3}{2}}( \breve{\beta} \rho) =0,\\
\left( \breve{\beta}^2 - \frac{4}{ \rho^2}\right) \left(
J'_{-\frac{3}{2}}( \breve{\beta} \rho) - \frac{3}{2 \rho}
J_{-\frac{3}{2}}( \breve{\beta} \rho) \right)^2 +
 \left( J''_{-\frac{3}{2}}( \breve{\beta} \rho) - \frac{1}{\rho} J'_{-\frac{3}{2}}( \breve{\beta} \rho) + \frac{3}{4 \rho^2} J_{-\frac{3}{2}}( \breve{\beta} \rho) \right)^2 = 0.
\end{gather*}
}
 \end{ex}

\section
{The Radially Holomorphic Potential in $\mathbb R^4$ and Geometric
Properties of Quaternionic M\"{o}bius Transformations with Real
Coefficients}

 When $\alpha = 2$, the system~\eqref{Bryukhov-4} in cylindrical
coordinates within Fueter's construction in $\mathbb
R^4$~\eqref{Fueter-4} is reduced to the Cauchy-Riemann type system
in the meridian half-plane (see, e.g.,
\cite{Aksenov:2005,GuHaSp:2008,ColSabStr:2009,ColSabStr:2016,BrKaeh:2016,Br:Hefei2020})
\begin{gather*}
\begin{cases}
\frac{\partial{u_0}}{\partial{x_0}} - \frac{\partial{u_{\rho}}}{\partial{\rho}}= 0,\\[1ex]
\frac{\partial{u_0}}{\partial{\rho}}=-\frac{\partial{u_{\rho}}}{\partial{x_0}},
\end{cases}
\end{gather*}
where  $ u_0 = \frac{\partial{g}}{\partial{x_0}}, \ u_{\rho} = -
\frac{\partial{g}}{\partial{\rho}}$.
 The
system~\eqref{2-axial-isotropic-system-4} is expressed as
\begin{gather}
\begin{cases}
        \frac{\partial{V_0}}{\partial{x_0}} + \frac{\partial{V_{\rho}}}{\partial{\rho}}  = 0,  \\
      \frac{\partial{V_0}}{\partial{\rho}} = \frac{\partial{V_{\rho}}}{\partial{x_0}}.
 \end{cases}
\label{Bryukhov-merid-vec}
\end{gather}
 The generalized Stokes-Beltrami system in the meridian half-plane~\eqref{generalized
Stokes-Beltrami} becomes the Cauchy-Riemann type system in the
meridian half-plane concerning functions $g = g(x_0, \rho)$,
$\hat{g} = \hat{g}(x_0, \rho)$:
\begin{gather}
\begin{cases}
\frac{\partial{g}}{\partial{x_0}}-\frac{\partial{\hat{g}}}{\partial{\rho}}=0,\\[1ex]
\frac{\partial{g}}{\partial{\rho}}=-\frac{\partial{\hat{g}}}{\partial{x_0}}.
\end{cases}
\label{Stokes-Beltrami-1}
\end{gather}
Generalized axially symmetric potential $g = g(x_0, \rho)$ and the
Stokes stream function $\hat{g} = \hat{g}(x_0, \rho)$  satisfy
equations
\begin{gather*}
\frac{\partial{^2}{g}}{\partial{x_0}^2} +  \frac{\partial
{^2}{g}}{\partial{\rho}^2} = 0, \quad
\frac{\partial{^2}{\hat{g}}}{\partial{x_0}^2} +  \frac{\partial
{^2}{\hat{g}}}{\partial{\rho}^2}  = 0.
\end{gather*}

On the other hand, an important concept of radially holomorphic
functions in $\mathbb R^4$ was introduced by G\"{u}rlebeck, Habetha,
Spr\"{o}{\ss}ig in 2008 in the context of the theory of
\emph{Holomorphic functions in n-dimensional space}
\cite{GuHaSp:2008}.

 \begin{defn}
 The radial differential operator in $\mathbb R^4$ is defined by formula
$$
 \partial_{rad}{G}  := \frac{1}{2} \left( \frac{\partial{}}{\partial{x_0}}- I \frac{\partial{}}{\partial{\rho}} \right) G := G'  \ \ \ \ \ (G = g + I \hat{g}).
$$

Every quaternion-valued function $G = g + I \hat{g}$ satisfying a
Cauchy-Riemann type differential equation  in $\Lambda$ $(\rho > 0)$
\begin{gather}
\overline{\partial}_{rad}{G} :=  \frac{1}{2} \left(
\frac{\partial{}}{\partial{x_0}} + I
\frac{\partial{}}{\partial{\rho}} \right) G = 0
\label{CRDO-con}
\end{gather}
is called a radially holomorphic in $\Lambda$, while
quaternion-conjugate function $\overline{G} = g - I \hat{g}$ is
called a radially anti-holomorphic in $\Lambda$.
 \end{defn}

 It is easy to see that general class of exact solutions of Eq.~\eqref{CRDO-con} is
equivalently represented as general class of exact solutions of the
Cauchy-Riemann type system in the meridian half-plane concerning
functions $g = g(x_0, \rho)$, $\hat{g} = \hat{g}(x_0,
\rho)$~\eqref{Stokes-Beltrami-1} (see, e.g., \cite{Br:Hefei2020}).

 The notation $\partial_{rad}{G} := G'$ has been justified in \cite{GuHaSp:2008} by some clear statements.
 In particular, Eq.~\eqref{CRDO-con} implies that
\begin{align}
G' = \frac{\partial{G}}{\partial{x_0}}.
\label{primitive}
\end{align}

 \begin{defn}
 Suppose that  a radially holomorphic  function $G = g + I \hat{g}$
in $\Lambda$ satisfies a differential equation
\begin{align*}
 G' = F,
\end{align*}
where $F = u_0 + I u_{\rho}$ characterizes a radially holomorphic
function in $\Lambda$. The function $G$ is called a radially
holomorphic primitive of $F$  in $\Lambda$.
\end{defn}

Elementary radially holomorphic functions in $\mathbb R^4$ may be
introduced as elementary functions of a quaternionic variable
 $ G = G(x) = g(x_0, \rho) + I \hat{g}_{\rho}(x_0, \rho)$ satisfying the following relations:
\begin{align*}
[x^{n} &:= r^{n}(\cos{n\varphi} + I \sin{n\varphi})]' = n x^{n-1};\\
[e^{x} &:= e^{x_0}(\cos{\rho}+I \sin{\rho})]'= e^{x};\\
[\cos{x} &:= \frac{1}{2}( e^{-Ix} + e^{Ix})]'= - \sin{x};\\
[\sin{x} &:= \frac{I}{2}( e^{-Ix} - e^{Ix})]'= \cos{x};\\
[\ln{x} &:= \ln r +I \varphi]' = x^{-1}.
\end{align*}

Applications of radially holomorphic functions in $\mathbb R^4$ have
been missed in the next book of G\"{u}rlebeck, Habetha,
Spr\"{o}{\ss}ig \cite{GuHaSp:2016}.
  Tools of the radially holomorphic potential in $\mathbb R^4$ in conjunction with tools of the theory
of \emph{Potential meridional vector fields in $\mathbb R^4$} allow
us to make up for the gap.

\begin{defn}
Radially holomorphic primitive $G = g + I \hat{g}$ in simply
connected open domains $\Lambda$ $ ( \rho > 0)$ into the framework
of the systems~\eqref{Bryukhov-merid-vec},~\eqref{Stokes-Beltrami-1}
is called the radially holomorphic potential in $\mathbb R^4$.
\end{defn}

When $\alpha =2$, the principal invariants of the Jacobian
matrix~\eqref{Jacobian-merid-4}
\begin{gather}
\begin{pmatrix}
  -\frac{\partial{V_{\rho}}}{\partial{\rho}} & \frac{\partial{V_{\rho}}}{\partial{x_0}} \frac{x_1}{\rho} &
 \frac{\partial{V_{\rho}}}{\partial{x_0}} \frac{x_2}{\rho} &  \frac{\partial{V_{\rho}}}{\partial{x_0}} \frac{x_3}{\rho}  \\[1ex]
\frac{\partial{V_{\rho}}}{\partial{x_0}} \frac{x_1}{\rho}  & \left(
\frac{\partial{V_{\rho}}}{\partial{\rho}} \frac{x_1^2}{\rho^2}  +
\frac{V_{\rho}}{\rho} \frac{x_2^2+x_3^2}{\rho^2}\right)  &
 \left( \frac{\partial{V_{\rho}}}{\partial{\rho}}- \frac{V_{\rho}}{\rho}\right)  \frac{x_1 x_2}{\rho^2} &
 \left( \frac{\partial{V_{\rho}}}{\partial{\rho}}- \frac{V_{\rho}}{\rho}\right)  \frac{x_1 x_3}{\rho^2}  \\[1ex]
          \frac{\partial{V_{\rho}}}{\partial{x_0}} \frac{x_2}{\rho}  & \left(
\frac{\partial{V_{\rho}}}{\partial{\rho}}-
\frac{V_{\rho}}{\rho}\right)  \frac{x_1 x_2}{\rho^2}  & \left(
\frac{\partial{V_{\rho}}}{\partial{\rho}} \frac{x_2^2}{\rho^2} +
\frac{V_{\rho}}{\rho} \frac{x_1^2+x_3^2}{\rho^2}\right) & \left(
\frac{\partial{V_{\rho}}}{\partial{\rho}}-
\frac{V_{\rho}}{\rho}\right)  \frac{x_2 x_3}{\rho^2}     \\[1ex]
          \frac{\partial{V_{\rho}}}{\partial{x_0}} \frac{x_3}{\rho}  & \left(
\frac{\partial{V_{\rho}}}{\partial{\rho}}-
\frac{V_{\rho}}{\rho}\right)  \frac{x_1 x_3}{\rho^2}   & \left(
\frac{\partial{V_{\rho}}}{\partial{\rho}}-
\frac{V_{\rho}}{\rho}\right)  \frac{x_2 x_3}{\rho^2} & \left(
\frac{\partial{V_{\rho}}}{\partial{\rho}} \frac{x_3^2}{\rho^2} +
\frac{V_{\rho}}{\rho} \frac{x_1^2+x_2^2}{\rho^2}\right)
 \label{Jacobian-merid-4-Fueter}
\end{pmatrix}
\end{gather}
are written as
\begin{gather*}
I_{\mathbf{J}(\vec V)} = 2 \frac{V_{\rho}}{\rho}, \quad
II_{\mathbf{J}(\vec V)} = - \left[ \left(
\frac{\partial{V_\rho}}{\partial{x_0}}  \right)^2 +  \left(
\frac{\partial{V_{\rho}}}{\partial{\rho}} \right)^2 \right] +
 \left( \frac{V_{\rho}}{ \rho} \right)^2, \\[1ex]
III_{\mathbf{J}(\vec V)} = -2 \frac{V_{\rho}}{\rho} \left[  \left(
\frac{\partial{V_\rho}}{\partial{x_0}}  \right)^2 +  \left(
\frac{\partial{V_{\rho}}}{\partial{\rho}} \right)^2 \right], \\
IV_{\mathbf{J}(\vec V)} = -\left( \frac{V_{\rho}}{ \rho} \right)^2
\left[ \left( \frac{\partial{V_\rho}}{\partial{x_0}} \right)^2 +
\left( \frac{\partial{V_{\rho}}}{\partial{\rho}} \right)^2 \right].
\end{gather*}

 \begin{cor}
 Roots of the characteristic equation~\eqref{characteristic lambda-4} of the Jacobian
matrix~\eqref{Jacobian-merid-4-Fueter} are  given by the formulas
\begin{gather}
 \lambda_{0,1} = \frac{V_{\rho}}{\rho}, \ \ \ \ \
\lambda_{2,3} = \pm  \sqrt{\left(
\frac{\partial{V_{\rho}}}{\partial{\rho}}\right)^2 + \left(
\frac{\partial{V_{\rho}}}{\partial{x_0}} \right)^2} =  \pm |F'|.
\label{lambda-merid-1}
\end{gather}
 \end{cor}

Exact formulas~\eqref{lambda-merid-1} demonstrate explicitly the
geometric specifics of the Jacobian
matrix~\eqref{Jacobian-merid-4-Fueter}.

Remarkable extensions of the basic concepts of M\"{o}bius
transformations in four dimensions (see, e.g., \cite {Ahlfors:1981,
Leut:CV17}) have been developed, in particular, by Porter
\cite{Porter:1998}, Cao \cite{Cao:2007}, Parker and Short
\cite{ParkerShort:CMFT2009}, Bisi and Gentili \cite{BisiGent:2009},
Stoppato \cite{Stoppato:2011}.

Meanwhile, geometric properties of quaternionic M\"{o}bius
transformations with real coefficients $a,b,c,d$ within Fueter's
construction in $\mathbb R^4$~\eqref{Fueter-4} $F(x) = (ax + b)(cx +
d)^{-1}$, where $ad - bc = 1$, may be characterized in the context
of the theory of \emph{Potential meridional vector fields in
$\mathbb R^4$}.

\begin{ex}
Consider quaternionic M\"{o}bius transformations with real
coefficients, where $c=1$, $ad - b = 1$.
  Radially anti-holomorphic function is expressed as $ \overline{F}(x) = - (\overline{x} +
d)^{-1} + a$. The radially holomorphic potential in $\mathbb R^4$ is
written as  $$G = - \ln{(x + d)} + ax.$$   We deal with analytic
models of meridional fields $\vec V= (V_0,\frac{x_1}{\rho}
V_{\rho},\frac{x_2}{\rho} V_{\rho},\frac{x_3}{\rho} V_{\rho})$,
where
\begin{align*}
V_0 = - \frac{x_0 + d}{(x_0 + d)^2 +\rho^2} + a, \quad V_{\rho}
= - \frac{\rho}{(x_0 + d)^2 + \rho^2},\\[1ex]
\frac{\partial{V_{\rho}}}{\partial{x_0}} = \frac{2\rho(x_0 +
d)}{[(x_0 + d)^2 + \rho^2]^2}, \quad
\frac{\partial{V_{\rho}}}{\partial{\rho}} = \frac{-(x_0 + d)^2 +
\rho^2}{[(x_0 + d)^2 + \rho^2]^2}.
\end{align*}

 The Jacobian
matrix~\eqref{Jacobian-merid-4-Fueter} takes the following form:
\begin{gather}
\begin{pmatrix}
 \frac{(x_0+ d)^2-\rho^2}{[(x_0+ d)^2+ \rho^2]^2} & \frac{2(x_0 + d)x_1}{[(x_0 + d)^2 + \rho^2]^2} &
 \frac{2(x_0 + d)x_2}{[(x_0 + d)^2 + \rho^2]^2} &  \frac{2(x_0 + d)x_3}{[(x_0 + d)^2 + \rho^2]^2}  \\[1ex]
\frac{2(x_0 + d)x_1}{[(x_0 + d)^2 + \rho^2]^2}  & \frac{-(x_0+
d)^2-\rho^2+2x_1^2}{[(x_0+ d)^2+ \rho^2]^2}  &
 \frac{2x_1 x_2}{[(x_0 + d)^2 + \rho^2]^2} &
 \frac{2x_1 x_3}{[(x_0 + d)^2 + \rho^2]^2}  \\[1ex]
    \frac{2(x_0 + d)x_2}{[(x_0 + d)^2 + \rho^2]^2}  & \frac{2x_1 x_2}{[(x_0 + d)^2 + \rho^2]^2}  & \frac{-(x_0+
d)^2-\rho^2+2x_2^2}{[(x_0+ d)^2+ \rho^2]^2} & \frac{2x_2 x_3}{[(x_0 + d)^2 + \rho^2]^2}  \\[1ex]
   \frac{2(x_0 + d)x_3}{[(x_0 + d)^2 + \rho^2]^2}  & \frac{2x_1 x_3}{[(x_0 + d)^2 + \rho^2]^2}  & \frac{2x_2 x_3}{[(x_0 + d)^2 + \rho^2]^2} & \frac{-(x_0+
d)^2-\rho^2+2x_3^2}{[(x_0+ d)^2+ \rho^2]^2}
 \label{Jacobian-merid-4-Moebius}
\end{pmatrix}
\end{gather}

Roots of the characteristic equation~\eqref{characteristic lambda-4}
of the Jacobian matrix~\eqref{Jacobian-merid-4-Moebius} are given by
the formulas
\begin{align*}
\lambda_{0,1} = -\frac{1}{(x_0 + d)^2 +\rho^2}, \quad \lambda_{2, 3}
= \pm \frac{1}{(x_0 + d)^2 +\rho^2}.
\end{align*}
Thus, the set of degenerate points of the Jacobian
matrix~\eqref{Jacobian-merid-4-Moebius} is empty.
 \end{ex}

 Geometric properties of elementary radially holomorphic functions in
$\mathbb R^4$ raise new issues for consideration.

 \section
{Analytic Models of Potential Meridional Fields in $\mathbb R^4$
Generated by Bessel Functions of the First Kind of Integer Order and
Quaternionic Argument}

 First outlines of the concept of integral transforms with quaternion-valued
 kernels within Fueter's construction in $\mathbb R^4$ were presented by the author in 2003
\cite{Br:2003}. Now tools of the Fourier-Fueter cosine and sine
transforms of real-valued original functions $\tilde{\eta} =
\tilde{\eta}(\tau)$ allow us to obtain integral representations for
Bessel functions of the first kind of integer order $n$ and
quaternionic argument $x$.

In accordance with the basic concepts of the theory of functions of
a complex variable and the one-sided Laplace transform (see, e.g.,
\cite{BatEr-IntTran,LavSh,BroSemMusMue:2015}), every complex-valued
original function $\tilde{\eta} = \tilde{\eta}(\tau)$ of a real
argument $\tau$ satisfies the following conditions:
\begin{enumerate}
  \item  the function $\tilde{\eta}(\tau)$ satisfies the H\"{o}lder's condition for
  any $\tau$ except for points $\tau = \tau_1,\tau_2,\ldots$, where
   the function $\tilde{\eta}(\tau)$ has discontinuities of the first kind
     (there exists a finite number of such points for every finite interval),
  \item the function $\tilde{\eta}(\tau) = 0$ for any $\tau<0$,
  \item $|\tilde{\eta}(\tau)| < M e^{s_0 \tau}$ with certain constants $M > 0, s_0 \geq 0$ for any $\tau \geq 0$.
  \end{enumerate}

In accordance with the H\"{o}lder's condition,
 there exist certain constants $A>0$, $0< \gamma \leq 1$, $\delta_0>0$ such that
 $|\tilde{\eta}(\tau+\delta)-\tilde{\eta}(\tau)|
\leq A|\delta|^{\gamma}$ for any $\delta$, where $|\delta| \leq
\delta_0$.
  The number $s_0$ is referred to as the growth rate of the function
  $\tilde{\eta}(\tau)$; $s_0$ may take the value 0 in case of bounded functions $\tilde{\eta}(\tau)$.

  To extend some concepts of integral transforms of real-valued original
  functions $\tilde{\eta} = \tilde{\eta}(\tau)$ within Fueter's construction in $\mathbb
R^4$~\eqref{Fueter-4}, a correspondence between two functions
$\tilde{\eta}(\tau)$ and $F(x)$
   in the form
\begin{gather*}
 F(x) := \int_{0}^{\infty} \tilde{\eta}(\tau) K(x,\tau) d\tau,
\end{gather*}
where quaternion-valued function $K(x,\tau)$ belongs to Fueter's
construction in $\mathbb R^4$ for any $\tau \geq 0$, shall
henceforth be referred to as integral transform with
quaternion-valued kernel $K(x,\tau)$.

 \begin{defn}
  Let $ \tilde{\eta}(\tau)$ be a real-valued original function and $\rho > 0$. Every integral transform with
quaternion-valued kernel $K(x,\tau) = e^{-x \tau}$
\begin{gather*}
 F(x) := \mathfrak{LF} \{ \tilde{\eta}(\tau) ; x \}
 = \int_{0}^{\infty} \tilde{\eta}(\tau) e^{-x \tau} d\tau
\end{gather*}
 is called the one-sided Laplace-Fueter transform of $
\tilde{\eta}(\tau)$.
 \end{defn}

Analytic properties of radially anti-holomorphic functions in
$\mathbb R^4$
\begin{align*}
  \overline{F}(x) = \int_{0}^{\infty} \tilde{\eta}(\tau) \overline{e^{-x \tau}} d\tau
= \int_{0}^{\infty} \tilde{\eta}(\tau) e^{-x_0\tau}[\cos(\rho\tau) +
I \sin(\rho\tau)]d\tau
 \end{align*}
are of particular interest to the theory of \emph{Potential
meridional vector fields in $\mathbb R^4$}.

 \begin{defn}
Let $ \tilde{\eta}(\tau)$ be a real-valued original function and
$\rho > 0$. Every integral transform with quaternion-valued kernel
$K(x,\tau) = \cos(x\tau)$
\begin{align*}
 F(x) := \mathfrak{FF}c \{ \tilde{\eta}(\tau) ; x \}
 = \int_{0}^{\infty} \tilde{\eta}(\tau) \cos(x\tau) d\tau = \frac{1}{2} \int_{0}^{\infty} \tilde{\eta}(\tau) ( e^{-Ix\tau} + e^{Ix\tau}) d\tau
 \end{align*}
is called the Fourier-Fueter cosine transform of
$\tilde{\eta}(\tau)$.
 \end{defn}

\begin{rem}
Consider an independent quaternionic variable of the following form:
$ y = Ix = -\rho + Ix_0$. The Fourier-Fueter cosine transform of $
\tilde{\eta}(\tau)$ may be equivalently represented by means of the
one-sided Laplace-Fueter transform of $ \tilde{\eta}(\tau)$:
\begin{align*}
 \mathfrak{FF}c \{ \tilde{\eta}(\tau) ; x \}
 = \frac{1}{2} [ \mathfrak{LF} \{ \tilde{\eta}(\tau) ; y \} + \mathfrak{LF} \{ \tilde{\eta}(\tau) ; -y \} ].
 \end{align*}
 \end{rem}

Analytic models of potential meridional fields generated by the
Fourier-Fueter cosine transform of $ \tilde{\eta}(\tau)$ are
described as $\vec V= (V_0,\frac{x_1}{\rho}
V_{\rho},\frac{x_2}{\rho} V_{\rho},\frac{x_3}{\rho} V_{\rho})$,
where
\begin{align*}
V_0 = \int_{0}^{\infty} \tilde{\eta}(\tau)
\cosh(\rho\tau)\cos(x_0\tau)d\tau,
 \quad
V_{\rho} = \int_{0}^{\infty} \tilde{\eta}(\tau)
\sinh(\rho\tau)\sin(x_0\tau)d\tau.
 \end{align*}

 \begin{defn}
Let $ \tilde{\eta}(\tau)$ be a real-valued original function and
$\rho > 0$. Every integral transform with quaternion-valued kernel
$K(x,\tau) = \sin(x\tau)$
\begin{align*}
 F(x) := \mathfrak{FF}s \{ \tilde{\eta}(\tau) ; x \}
  = \int_{0}^{\infty} \tilde{\eta}(\tau) \sin(x\tau) d\tau = \frac{I}{2} \int_{0}^{\infty} \tilde{\eta}(\tau) ( e^{-Ix\tau} - e^{Ix\tau}) d\tau
 \end{align*}
is called the Fourier-Fueter sine transform of $
\tilde{\eta}(\tau)$.
 \end{defn}

 \begin{rem}
 The Fourier-Fueter sine transform of $ \tilde{\eta}(\tau)$ may be equivalently represented by
means of the one-sided Laplace-Fueter transform of $
\tilde{\eta}(\tau)$:
\begin{align*}
 \mathfrak{FF}s \{ \tilde{\eta}(\tau) ; x \}
 = \frac{I}{2} [ \mathfrak{LF} \{ \tilde{\eta}(\tau) ; y \} - \mathfrak{LF} \{ \tilde{\eta}(\tau) ; -y \}].
 \end{align*}
 \end{rem}

 Analytic models of potential meridional fields generated by
the Fourier-Fueter sine transform of $ \tilde{\eta}(\tau)$ are
described as $\vec V= (V_0,\frac{x_1}{\rho}
V_{\rho},\frac{x_2}{\rho} V_{\rho},\frac{x_3}{\rho} V_{\rho})$,
where
\begin{align*}
V_0 = \int_{0}^{\infty} \tilde{\eta}(\tau)
\cosh(\rho\tau)\sin(x_0\tau)d\tau, \quad V_{\rho} = -
\int_{0}^{\infty} \tilde{\eta}(\tau)
\sinh(\rho\tau)\cos(x_0\tau)d\tau.
 \end{align*}

Important properties of quaternionic power series with real
coefficients studied by Leutwiler and Hempfling in the context of
\emph{Modified quaternionic analysis in $\mathbb R^4$} (see, e.g.,
\cite{Leut:CV17,HempLeut:1996}) allow us to extend every Bessel
function of the first kind of integer order and complex argument
(see, e.g., \cite{Watson:1944,BatEr-Higher-II,BatEr-IntTran}) within
Fueter's construction in $\mathbb R^4$ ~\eqref{Fueter-4} from the
disk of radius $r$ in the complex plane $D_r = \{ (x_0,x_1): x_0^2 +
x_1^2 < r^2 \} $ to the ball of radius $r$ in $\mathbb R^4$ $B_r^4 =
\{ (x_0,x_1,x_2,x_3): x_0^2 + x_1^2 + x_2^2 + x_3^2 < r^2 \} $. As a
corollary, every Bessel function of the first kind of integer order
$n$ and quaternionic argument is represented by a quaternionic power
series with real coefficients:
\begin{align*}
J_{n}(x) = \sum_{m = 0}^\infty \frac{(-1)^m }{m!(n+m)!} \left(
\frac{x}{2} \right)^{n+2m}.
\end{align*}
Furthermore, every elementary radially holomorphic function $\left(
\frac{x}{2} \right)^{m}$ may be expanded in a series of Bessel
functions of the first kind of integer order $(m+2n)$ and
quaternionic argument:
\begin{align*}
 \left( \frac{x}{2} \right)^{m} = \sum_{n = 0}^\infty \frac{(m+2n)(m+n-1)!}{n!} J_{m+2n}(x) \quad (m = 1,2,3,\ldots)
\end{align*}

 Using the Fourier-Fueter cosine transform of $\ \tilde{\eta}(\tau) = \frac{\cos (2n \arccos
\tau )}{\sqrt{1 - \tau^2}}$, integral representations for Bessel
functions of the first kind of even integer order and quaternionic
argument may be obtained:
\begin{align*}
\frac{\pi}{2} (-1)^n J_{2n}(x)
  = \mathfrak{FF}c \{ \tilde{\eta}(\tau) ; x \}
 = \int_{0}^{1} \frac{ \cos (2n \arccos \tau )}{\sqrt{1 - \tau^2}} \cos(x\tau)
 d\tau.
\end{align*}

Using the Fourier-Fueter sine transform of $\tilde{\eta}(\tau) =
\frac{\cos [(2n+1) \arccos \tau ]}{\sqrt{1 - \tau^2}}$, integral
representations for Bessel functions of the first kind of odd
integer order and quaternionic argument may be obtained:
\begin{align*}
\frac{\pi}{2} (-1)^n J_{2n+1}(x)
  = \mathfrak{FF}s \{ \tilde{\eta}(\tau) ; x \}
 = \int_{0}^{1} \frac{ \cos [(2n+1) \arccos \tau ]}{\sqrt{1 - \tau^2}} \sin(x\tau)
 d\tau.
\end{align*}

\begin{ex}
Bessel function of the first kind of order zero and quaternionic
argument is represented as $J_{0}(x) = \sum_{m = 0}^\infty
\frac{(-1)^m }{(m!)^2} \left( \frac{x}{2} \right)^{2m}.$

Integral representation of Bessel function of the first kind of
order zero and quaternionic argument is given by the formula
\begin{gather*}
 J_{0}(x)
  = \frac{2}{\pi} \mathfrak{FF}c \{ \tilde{\eta}(\tau) ; x \}
 = \frac{2}{\pi} \int_{0}^{1} \frac{1}{\sqrt{1 - \tau^2}}\cos(x\tau)
 d\tau.
\end{gather*}

 As seen, remarkable analytic model of a potential meridional field generated by Bessel
function of the first kind of order zero and quaternionic argument
is described as $\vec V= (V_0,\frac{x_1}{\rho}
V_{\rho},\frac{x_2}{\rho} V_{\rho},\frac{x_3}{\rho} V_{\rho})$,
where
\begin{align*}
V_0 = \int_{0}^{1} \frac{\cosh(\rho\tau)}{\sqrt{1 - \tau^2}}
\cos(x_0\tau)d\tau, \quad  V_{\rho} = \int_{0}^{1}
\frac{\sinh(\rho\tau)}{\sqrt{1 - \tau^2}} \sin(x_0\tau)d\tau.
\end{align*}
 \end{ex}

\section{Four-Dimensional Harmonic Meridional Mappings of the Second Kind
and Gradient Dynamical Systems with Harmonic Potential}

 Some remarkable properties of the Hessian matrix $\mathbf{H}(h(x))$ of harmonic
potentials $h= h(x_0,x_1,x_2,x_3)$ in the general four-dimensional
setting were described by Yanushauskas in 1980 by means of
homogeneous harmonic polynomials \cite{Yanush:1980}.

The geometric specifics of the Hessian matrix $\mathbf{H}(h(x))$
within harmonic potential meridional fields $\vec V= \mathrm{grad} \
h$ may be efficiently studied using the Jacobian
matrix~\eqref{Jacobian-merid-4} in case $\alpha = 0$:
\begin{gather}
\tiny{
\begin{pmatrix}
 \left( -\frac{\partial{V_{\rho}}}{\partial{\rho}} - 2\frac{V_{\rho}}{\rho} \right) & \frac{\partial{V_{\rho}}}{\partial{x_0}} \frac{x_1}{\rho} &
 \frac{\partial{V_{\rho}}}{\partial{x_0}} \frac{x_2}{\rho} &  \frac{\partial{V_{\rho}}}{\partial{x_0}} \frac{x_3}{\rho}  \\[1ex]
\frac{\partial{V_{\rho}}}{\partial{x_0}} \frac{x_1}{\rho}  & \left(
\frac{\partial{V_{\rho}}}{\partial{\rho}} \frac{x_1^2}{\rho^2}  +
\frac{V_{\rho}}{\rho} \frac{x_2^2+x_3^2}{\rho^2}\right)  &
 \left( \frac{\partial{V_{\rho}}}{\partial{\rho}}- \frac{V_{\rho}}{\rho}\right)  \frac{x_1 x_2}{\rho^2} &
 \left( \frac{\partial{V_{\rho}}}{\partial{\rho}}- \frac{V_{\rho}}{\rho}\right)  \frac{x_1 x_3}{\rho^2}  \\[1ex]
          \frac{\partial{V_{\rho}}}{\partial{x_0}} \frac{x_2}{\rho}  & \left(
\frac{\partial{V_{\rho}}}{\partial{\rho}}-
\frac{V_{\rho}}{\rho}\right)  \frac{x_1 x_2}{\rho^2}  & \left(
\frac{\partial{V_{\rho}}}{\partial{\rho}} \frac{x_2^2}{\rho^2} +
\frac{V_{\rho}}{\rho} \frac{x_1^2+x_3^2}{\rho^2}\right) & \left(
\frac{\partial{V_{\rho}}}{\partial{\rho}}-
\frac{V_{\rho}}{\rho}\right)  \frac{x_2 x_3}{\rho^2}     \\[1ex]
          \frac{\partial{V_{\rho}}}{\partial{x_0}} \frac{x_3}{\rho}  & \left(
\frac{\partial{V_{\rho}}}{\partial{\rho}}-
\frac{V_{\rho}}{\rho}\right)  \frac{x_1 x_3}{\rho^2}   & \left(
\frac{\partial{V_{\rho}}}{\partial{\rho}}-
\frac{V_{\rho}}{\rho}\right)  \frac{x_2 x_3}{\rho^2} & \left(
\frac{\partial{V_{\rho}}}{\partial{\rho}} \frac{x_3^2}{\rho^2} +
\frac{V_{\rho}}{\rho} \frac{x_1^2+x_2^2}{\rho^2}\right)
 \label{Jacobian-merid-4-0}
\end{pmatrix}.
}
\end{gather}

Let us look at the basic properties of harmonic potential meridional
 fields in $\mathbb R^4$, where the systems~\eqref{A_4^alpha
system-meridional},~\eqref{Bryukhov-merid-4} are expressed as
\begin{gather*}
\begin{cases}
         \rho \left( \frac{\partial{u_0}}{\partial{x_0}} - \frac{\partial{u_{\rho}}}{\partial{\rho}} \right) - 2 u_{\rho} = 0,  \\
      \frac{\partial{u_0}}{\partial{\rho}}=-\frac{\partial{u_{\rho}}}{\partial{x_0}},
\end{cases}
\end{gather*}
\begin{gather*}
\begin{cases}
         \rho \left( \frac{\partial{V_0}}{\partial{x_0}} + \frac{\partial{V_{\rho}}}{\partial{\rho}} \right) + 2 V_{\rho} = 0,  \\
      \frac{\partial{V_0}}{\partial{\rho}} =
      \frac{\partial{V_{\rho}}}{\partial{x_0}},
\end{cases}
\end{gather*}
\begin{gather*}
\rho \left( \frac{\partial{^2}{g}}{\partial{x_0}^2} + \frac{\partial
{^2}{g}}{\partial{\rho}^2} \right) + 2
\frac{\partial{g}}{\partial{\rho}} = 0, \quad \quad \rho \left(
\frac{\partial{^2}{\hat{g}}}{\partial{x_0}^2} + \frac{\partial
{^2}{\hat{g}}}{\partial{\rho}^2} \right) - 2
\frac{\partial{\hat{g}}}{\partial{\rho}} = 0.
\end{gather*}

 The characteristic equation $(\ref{characteristic lambda-4})$ in case $\alpha = 0$ is written as incomplete equation:
\begin{gather}
 \lambda^4 + II_{\mathbf{J}(\vec V)} \lambda^2 - III_{\mathbf{J}(\vec V)} \lambda + IV_{\mathbf{J}(\vec V)} = 0,
\label{character lambda-4}
\end{gather}
where
\begin{align*}
II_{\mathbf{J}(\vec V)} &= - \left[ \left(
\frac{\partial{V_\rho}}{\partial{x_0}}  \right)^2 +  \left(
\frac{\partial{V_{\rho}}}{\partial{\rho}} \right)^2 \right]
 - 2 \frac{V_{\rho}}{\rho} \frac{\partial{V_{\rho}}}{\partial{\rho}} - 3 \left( \frac{V_{\rho}}{ \rho} \right)^2,\\[1ex]
 III_{\mathbf{J}(\vec V)}
&= - 2 \frac{V_{\rho}}{\rho} \left[ \left(
\frac{\partial{V_\rho}}{\partial{x_0}}  \right)^2 +  \left(
\frac{\partial{V_{\rho}}}{\partial{\rho}} \right)^2 \right]
 - 2 \left( \frac{V_{\rho}}{ \rho} \right)^2 \left( 2 \frac{\partial{V_{\rho}}}{\partial{\rho}} + \frac{V_{\rho}}{ \rho} \right),\\[1ex]
IV_{\mathbf{J}(\vec V)} &= - \left( \frac{V_{\rho}}{ \rho} \right)^2
\left[ \left( \frac{\partial{V_\rho}}{\partial{x_0}} \right)^2 +
\left( \frac{\partial{V_{\rho}}}{\partial{\rho}} \right)^2 \right]
 - 2 \left( \frac{V_{\rho}}{ \rho} \right)^3 \frac{\partial{V_{\rho}}}{\partial{\rho}}.
\end{align*}

\begin{cor}
Roots of the characteristic equation~\eqref{character lambda-4} are
given by the formulas
\begin{gather}
 \lambda_{0,1} = \frac{V_{\rho}}{\rho}, \; \lambda_{2,3} = - \frac{V_{\rho}}{\rho}  \pm
 \sqrt{\left(
\frac{V_{\rho}}{\rho} + \frac{\partial{V_{\rho}}}{\partial{\rho}}
\right)^2 + \left(
\frac{\partial{V_{\rho}}}{\partial{x_0}}\right)^2}.
\label{incomplete lambda-4}
\end{gather}
The set of degenerate points of the Jacobian
matrix~\eqref{Jacobian-merid-4-0} is provided by two independent
equations:
\begin{gather}
{V_{\rho}}=0,\quad
\left(\frac{\partial{V_{\rho}}}{\partial{x_0}}\right)^2
+\left(\frac{\partial{V_{\rho}}}{\partial{\rho}}\right)^2 + 2
\frac{V_{\rho}}{\rho}\frac{\partial{V_{\rho}}}{\partial{\rho}}=0.
\label{incomplete degenerate-4}
\end{gather}
\end{cor}

Exact formulas~\eqref{incomplete lambda-4},~\eqref{incomplete
degenerate-4} demonstrate explicitly the geometric specifics of the
Jacobian matrix~\eqref{Jacobian-merid-4-0}.

\begin{defn}
Let $\Lambda \subset \mathbb R^4$ be a simply connected open domain,
where \\ $x_1 \neq 0, x_2 \neq 0, x_3 \neq 0$. Assume that an exact
solution $(u_0, u_1, u_2, u_3)$ of the system $(R_4)$ satisfies
axially symmetric conditions ${u_1}{x_2}={u_2}{x_1}, \quad
{u_1}{x_3}={u_3}{x_1}, \quad {u_2}{x_3}={u_3}{x_2}$ in $\Lambda$.
Then mapping $u = u_0 + iu_1 + ju_2 + ku_3: \Lambda \rightarrow
\mathbb{R}^4$ is called four-dimensional harmonic meridional mapping
of the first kind, while mapping $ \overline{u} = u_0 - iu_1 - ju_2
- ku_3: \Lambda \rightarrow \mathbb{R}^4$ is called four-dimensional
harmonic meridional mapping of the second kind.
\end{defn}

The set of degenerate points of every four-dimensional harmonic
meridional mapping of the second kind coincides with the set of
degenerate points of the Jacobian matrix~\eqref{Jacobian-merid-4-0}.

\begin{ex}
Consider a generalized axially symmetric potential in case $\alpha
=0$ using Bessel function of the first kind of order $-\frac{1}{2}$
and real argument $\breve{\beta} \rho$:
 $$
 g(x_0,  \rho) = e^{\breve{\beta} x_0} \rho^{-\frac{1}{2}} J_{-\frac{1}{2}}( \breve{\beta}
 \rho).
 $$
 We deal with analytic models of harmonic potential meridional fields
 $\vec V= (V_0,\frac{x_1}{\rho} V_{\rho},\frac{x_2}{\rho} V_{\rho},\frac{x_3}{\rho} V_{\rho})$, where
 \begin{align*}
V_0 = \breve{\beta} e^{\breve{\beta} x_0} \rho^{-\frac{1}{2}}
J_{-\frac{1}{2}}(\breve{\beta} \rho), \quad
 V_{\rho} = e^{\breve{\beta} x_0}
 \rho^{-\frac{1}{2}} \left( J'_{-\frac{1}{2}}( \breve{\beta} \rho) - \frac{1}{2 \rho} J_{-\frac{1}{2}}( \breve{\beta} \rho)
 \right),
 \end{align*}
 such that
 \begin{align*}
\frac{\partial{V_{\rho}}}{\partial{x_0}} &= \breve{\beta}
e^{\breve{\beta} x_0} \rho^{-\frac{1}{2}} \left( J'_{-\frac{1}{2}}(
\breve{\beta} \rho) - \frac{1}{2 \rho} J_{-\frac{1}{2}}(
\breve{\beta} \rho) \right),\\
\frac{\partial{V_{\rho}}}{\partial{\rho}} &= e^{\breve{\beta} x_0}
\rho^{-\frac{1}{2}} \left( J''_{-\frac{1}{2}}( \breve{\beta} \rho)-
\frac{1}{\rho} J'_{-\frac{1}{2}}( \breve{\beta} \rho) + \frac{3}{4
\rho^2} J_{-\frac{1}{2}}( \breve{\beta} \rho) \right).
\end{align*}

Roots of the characteristic equation~\eqref{character lambda-4} are
given by these formulas:
 \small{
\begin{align*}
&\lambda_{0,1} = \frac{e^{\breve{\beta} x_0}}{\rho \sqrt{\rho}}
\left( J'_{-\frac{1}{2}}( \breve{\beta} \rho) - \frac{1}{2 \rho}
J_{-\frac{1}{2}}( \breve{\beta} \rho)
 \right),\\
&\lambda_{2,3}= - \frac{e^{\breve{\beta} x_0}}{\rho \sqrt{\rho}}
\left( J'_{-\frac{1}{2}}( \breve{\beta} \rho) - \frac{1}{2 \rho}
J_{-\frac{1}{2}}( \breve{\beta} \rho)
 \right) \pm \\
&\frac{e^{\breve{\beta} x_0}}{ \sqrt{\rho}} \sqrt{ \breve{\beta}^2
\left( J'_{-\frac{1}{2}}( \breve{\beta} \rho) - \frac{1}{2 \rho}
J_{-\frac{1}{2}}( \breve{\beta} \rho) \right)^2 + \left(
J''_{-\frac{1}{2}}( \breve{\beta} \rho) + \frac{1}{4 \rho^2}
J_{-\frac{1}{2}}( \breve{\beta} \rho) \right)^2}.
\end{align*}
 }
 The set of degenerate points of the Jacobian
matrix~\eqref{Jacobian-merid-4-0} is provided by two independent
equations:
 \small{
\begin{gather*}
J'_{-\frac{1}{2}}( \breve{\beta} \rho) - \frac{1}{2 \rho} J_{-\frac{1}{2}}( \breve{\beta} \rho) =0,\\
\left( \breve{\beta}^2 - \frac{1}{ \rho^2}\right) \left(
J'_{-\frac{1}{2}}( \breve{\beta} \rho) - \frac{1}{2 \rho}
J_{-\frac{1}{2}}( \breve{\beta} \rho) \right)^2 + \left(
J''_{-\frac{1}{2}}( \breve{\beta} \rho) + \frac{1}{4 \rho^2}
J_{-\frac{1}{2}}( \breve{\beta} \rho) \right)^2 = 0.
\end{gather*}
 }
 \end{ex}

Roots  of the characteristic equation~\eqref{character lambda-4} may
be interpreted as the Poincar\'{e} stability coefficients  in the
context of the theory of \emph{Gradient dynamical systems with
harmonic potential}, where $\vec V= \frac {d{\vec x}}{dt} =
\mathrm{grad} \ h$, $ \ \Delta \ h= 0$. It should be noted that
properties of the sets of positive, zero and negative values of the
roots $\lambda_0, \lambda_1,\lambda_2,\lambda_3$ are of particular
interest to \emph{Stability theory} (see, e.g.,
\cite{Chetayev:1961,Perko:2001,Kozlov:1993,KozlovFurta:2001}).

\end{document}